\documentclass[psamsfonts]{amsart}

\usepackage{xcolor}
\usepackage{colortbl,booktabs}
\usepackage{setspace}
\usepackage{tipa}
\usepackage{mathrsfs}
\usepackage{amsfonts}
\usepackage{amssymb,amsfonts}
\usepackage[all,arc]{xy}
\usepackage{enumerate}
\usepackage{mathrsfs}
\usepackage{graphicx}
\def\dis{\displaystyle}

\newtheorem{thm}{Theorem}[section]
\newtheorem{cor}[thm]{Corollary}
\newtheorem{prop}[thm]{Proposition}
\newtheorem{lem}[thm]{Lemma}

\theoremstyle{definition}
\newtheorem{defn}[thm]{Definition}

\theoremstyle{remark}
\newtheorem{rem}[thm]{Remark}

\numberwithin{equation}{section}

\begin{document}

\date{}
\date{}
\title{On mean curvature flow with driving force starting as singular initial hypersurface}

\author{Ryunosuke Mori ${}^1$, Longjie ZHANG ${}^2$}

\date{September, 2016\\
${}^1$ 
Corresponding author University:Graduate School of Mathematical Sciences, The University of Tokyo. Address:3-8-1 Komaba Meguro-ku Tokyo 153-8914, Japan. Email:moriryu@ms.u-tokyo.ac.jp, 45c136045@gmail.com\\
${}^2$ 
Corresponding author University:Graduate School of Mathematical Sciences, The University of Tokyo. Address:3-8-1 Komaba Meguro-ku Tokyo 153-8914, Japan. Email:zhanglj@ms.u-tokyo.ac.jp, zhanglj919@gmail.com
\\
The second author is the Research Fellow of Japan Society for the Promotion of Science, Number: 17J05160.}

\maketitle

\begin{minipage}{140mm}

{{\bf Abstract:} We consider an axisymmetric closed hypersurface evolving by its mean curvature with driving force under singular initial hypersurface. We study this problem by level set method. We give some criteria to judge whether the interface evolution is fattening or non-fattening. 

 }

{\bf 2010MSC:} 35A01, 35A02, 35K55, 53C44.

\end{minipage}

$$$$

\section{Introduction}\large

 This paper studies the mean curvature flow with driving force under singular initial data
\begin{equation}\label{eq:cur}
V=-\kappa+A\, \ \textrm{on}\ \Gamma(t)\subset \mathbb{R}^{n+1} ,
\end{equation}
\begin{equation}\label{eq:initial1}
\Gamma(0)=\Gamma_0,
\end{equation}
where $\Gamma(t)$ is a smooth family of hypersurfaces embedded in $\mathbb{R}^{n+1}$, $V$ is the outer normal velocity of $\Gamma(t)$, $\kappa$ is the mean curvature of $\Gamma(t)$ and $A>0$, called driving force, is a constant. Here the sign of $\kappa$ is taken so that the problem is parabolic. For example, under the definition, the mean curvature of unit sphere in $\mathbb{R}^{n+1}$ is $n$.

\begin{figure}[htbp]
	\begin{center}
            \includegraphics[height=8.0cm]{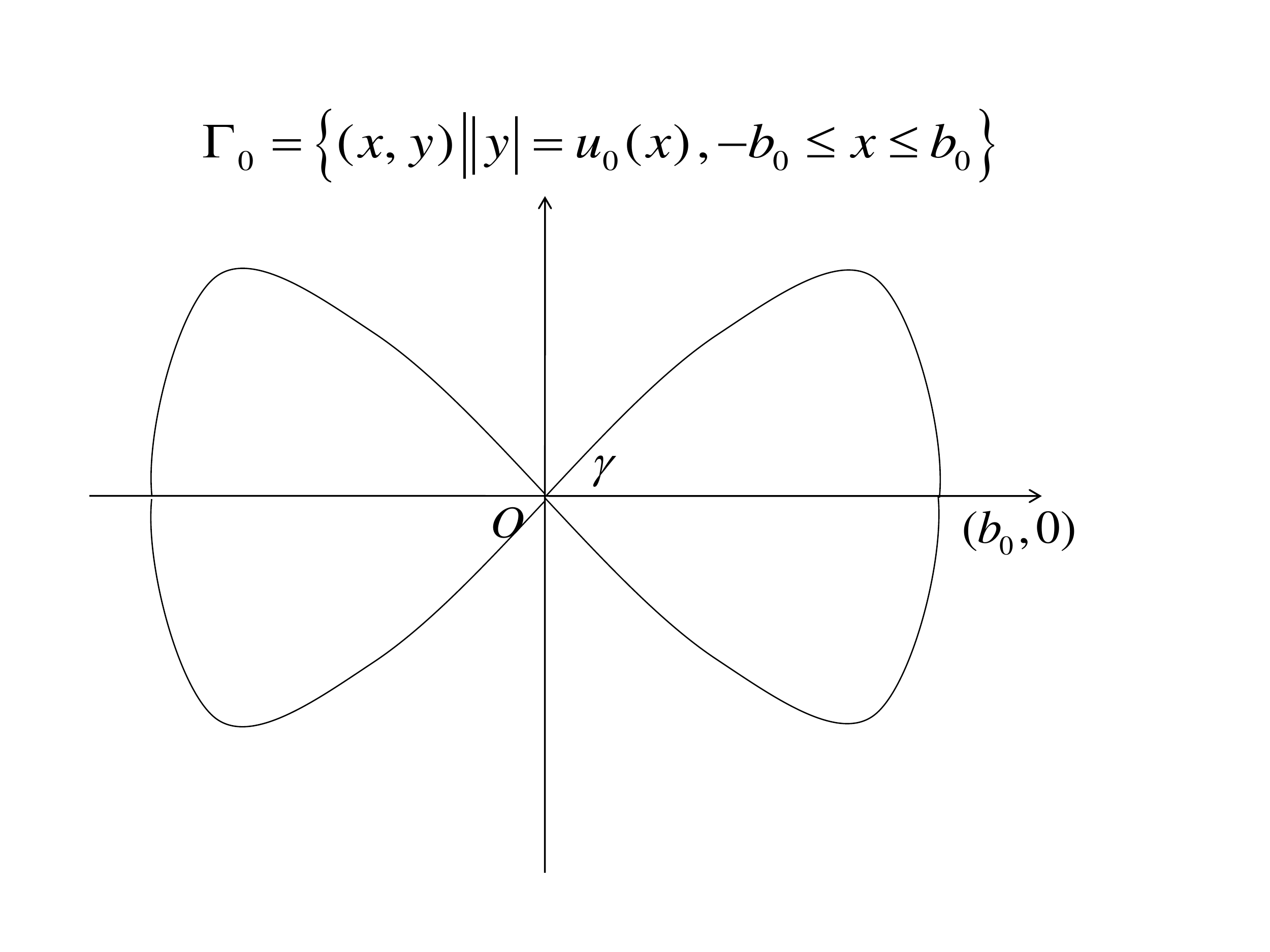}
		\vskip 0pt
		\caption{Initial curve $\Gamma_0$}
        \label{fig:u03}
	\end{center}
\end{figure}

In this paper we consider that the initial data $\Gamma_0$ is smooth except for origin and can be wrriten into 
\begin{equation}\label{eqn:initialcurveg}
\Gamma_0=\{(x,y)\in\mathbb{R}\times\mathbb{R}^{n}\mid |y|=u_0(x), -b_0\leq x\leq b_0\},
\end{equation}
where $u_0(x)$ is even and satisfies $u_0(x)>0$, for $x\in (-b_0,0)\cup (0,b_0)$, $u_0(0)=u_0(-b_0)=u_0(b_0)=0$. Since $\Gamma_0$ is smooth at $(-b,0,\cdots,0)$ and $(b,0,\cdots,0)$, it is easy to see that $u_{0}^{\prime}(-b_0)=-u_{0}^{\prime}(b_0)=+\infty$.

Let 
\begin{equation}\label{eqn:contactangle}
\gamma:=\lim\limits_{x\rightarrow 0^+}\arctan u_0^{\prime}(x)\in[0,\pi/2],
\end{equation}
seeing Figure \ref{fig:u03}.

{\bf Main assumptions for $\gamma=\pi/2$.}
Under the condition $\gamma=\pi/2$, let $\Lambda_0=\Gamma_0\cap \{(x,y)\in \mathbb{R}\times\mathbb{R}^n\mid x\geq0\}$.

We consider another problem.
\begin{equation}
V=-\kappa+A\, \ \textrm{on}\ \Lambda^+(t)\subset \mathbb{R}^{n+1},\tag{\ref{eq:cur}*}
\end{equation}
\begin{equation}
\Lambda^+(0)=\Lambda_0\tag{\ref{eq:initial1}*}
\end{equation}
(Figure \ref{fig:u02}).
\begin{figure}[htbp]
	\begin{center}
            \includegraphics[height=8.0cm]{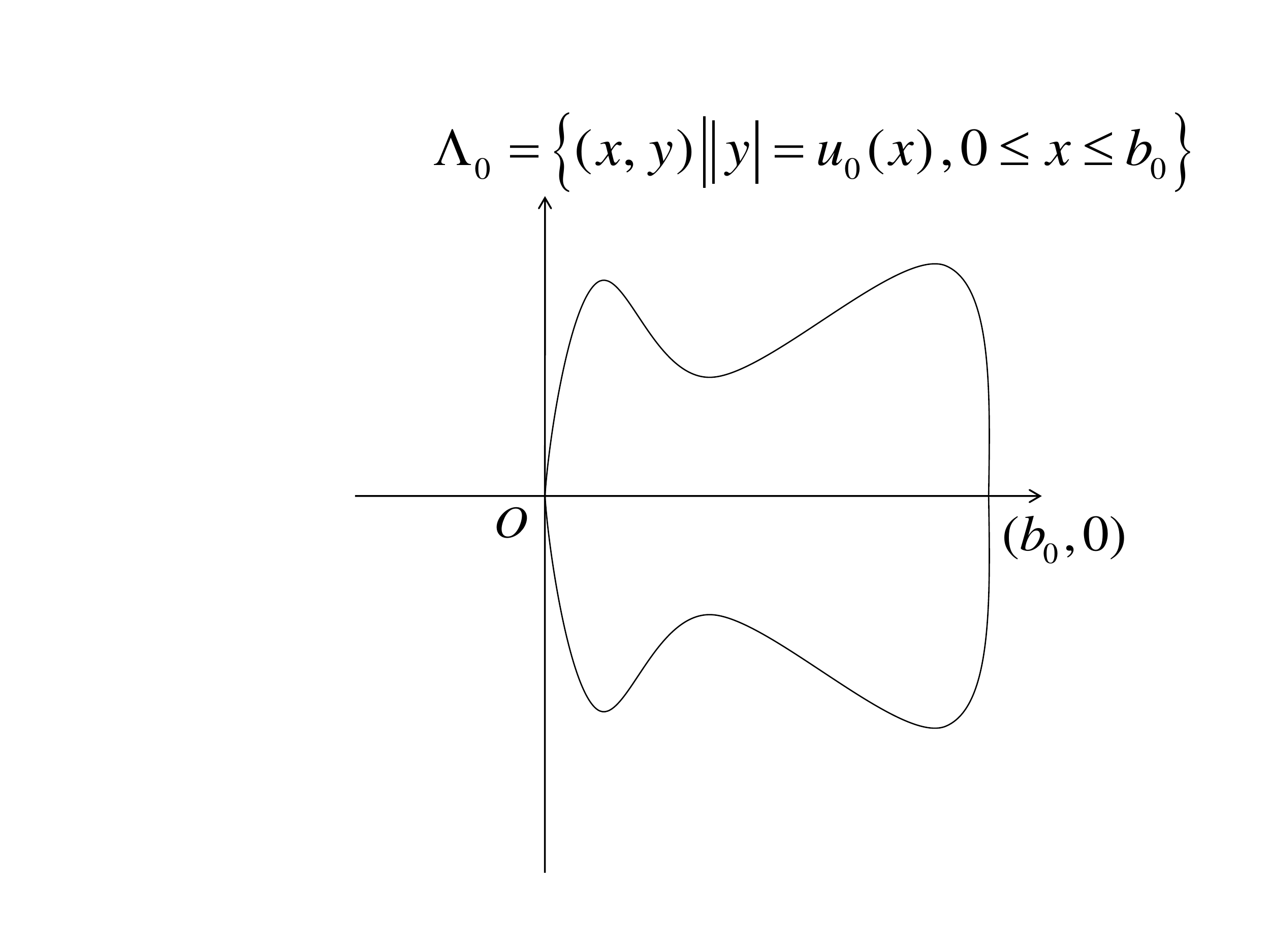}
		\vskip 0pt
		\caption{Initial curve $\Lambda_0$ for $\gamma=\pi/2$}
        \label{fig:u02}
	\end{center}
\end{figure}
We consider this problem by level set method. By the theory in \cite{G}, there exists a unique viscosity solution $\phi$ of the following level set equation
$$
\left\{
\begin{array}{lcl}
\dis{\phi_t=|\nabla \phi|\textmd{div}(\frac{\nabla \phi}{|\nabla \phi|})+A|\nabla \phi|}\ \textrm{in}\  \mathbb{R}^{n+1}\times(0,T),\\
\phi(x,y,0)=a_1(x,y),
\end{array}
\right.
$$
where $a_1(x,y)$ satisfies $\Lambda_0=\{(x,y)\mid a_1(x,y)=0\}$ and $\{(x,y)\mid a_1(x,y)>0\}$ is bounded. The results in appendix show that the zero set of $\phi$ is not fattening in a short time. Indeed, thanks to Theorem \ref{thm:gu}, the zero set of $\phi$ can be written into
$$
\Lambda^+(t)=\{(x,y)\mid \phi(x,y,t)=0\}=\{(x,y)\in\mathbb{R}^{n+1}\mid |y|=v(x,t), a_*(t)\leq x\leq b_*(t)\},\ 0<t<T_*.
$$ 
Moreover, $(v,a_*,b_*)$ is the solution of the following free boundary problem
\begin{equation}
\left\{
\begin{array}{lcl}
\dis{u_t=\frac{u_{xx}}{1+u_x^2}-\frac{n-1}{u}+A\sqrt{1+u_x^2}},\ x\in(a_*(t),b_*(t)),\ 0<t< T_*,\\
u(a_*(t),t)=0,\ u(b_*(t),t)=0,\ 0\leq t< T_*,\\
u_x(a_*(t),t)=\infty,\ u_x(b_*(t),t)=-\infty,\ 0\leq t<T_*,\\
u(x,0)=u_0(x),\ 0\leq x\leq b_0,\\
u(x,t)>0,\ x\in(a_*(t),b_*(t)),\ 0<t< T_*.
\end{array}
\right.\tag{*}
\end{equation}
In this paper, $a_*$ and $b_*$ are called the end points of $\Lambda^+(t)$.

$Assumption$ $(A+)$: There exists $\delta>0$ such that $a_*(t)\geq0$ for $0\leq t<\delta$.

$Assumption$ $(A-)$: There exists $\delta>0$ such that $a_*(t)<0$ for $0<t<\delta$.

Here we give some sufficient conditions such that main assumptions hold.

If $\kappa(O)<A$, then, since 
$$
a_*^{\prime}(0)=\kappa(O)-A<0,
$$
$a_*(t)<0$ for any small $t>0$. Similarly, if $\kappa(O)>A$, $a(t)>0$ for any small $t>0$. Here 
$$
\kappa(O)=\lim\limits_{x\rightarrow 0^+}\left(-\frac{u_{xx}}{(1+u_x^2)^{3/2}}+\frac{n-1}{u\sqrt{1+u_x^2}}\right)
$$
denotes the mean curvature of $\Lambda_0$ at the origin $O$. 

\begin{figure}[htbp]
	\begin{center}
            \includegraphics[height=8.0cm]{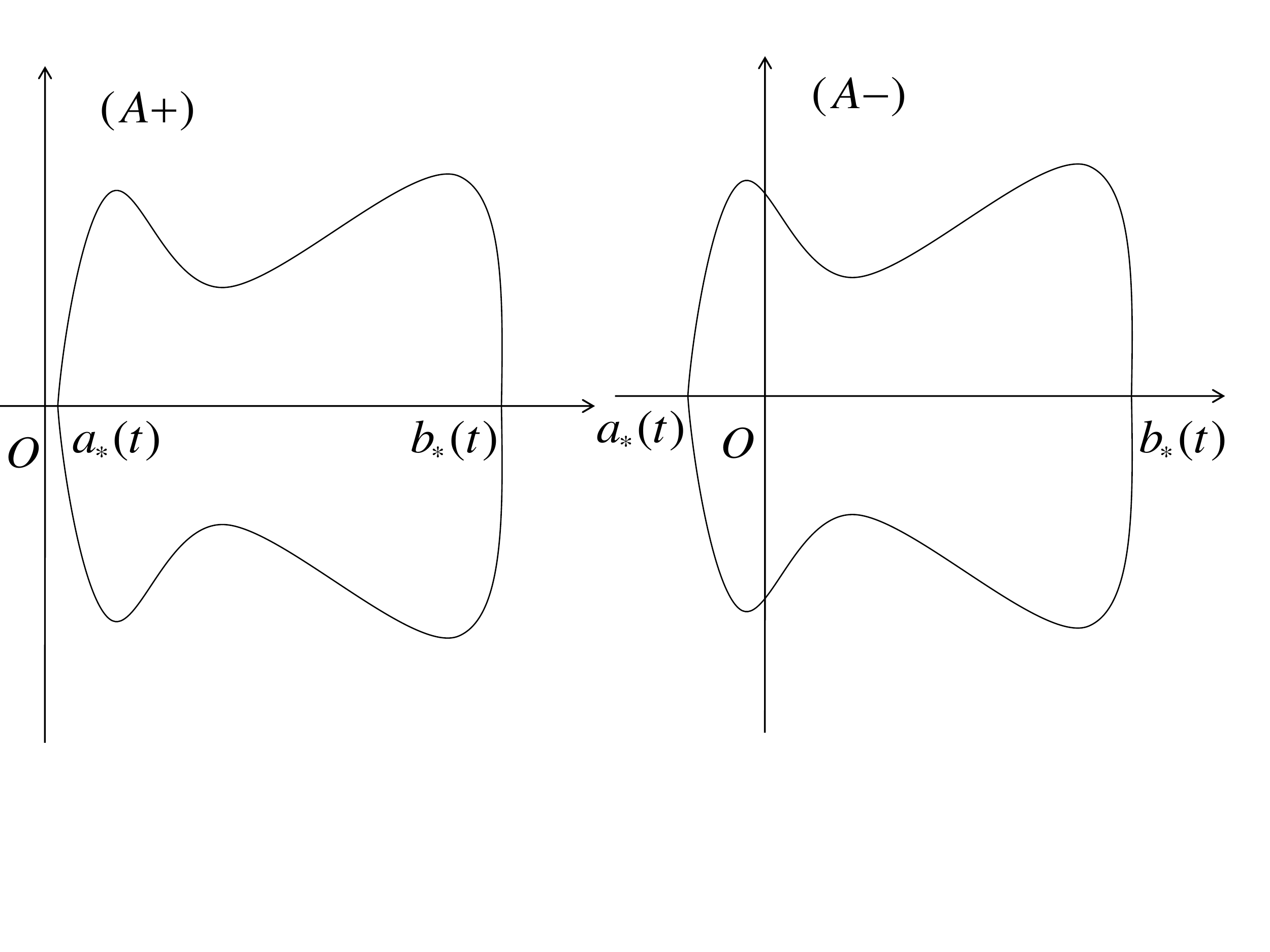}
		\vskip 0pt
		\caption{Main assumptions for $\gamma=\pi/2$}
        \label{fig:apm}
	\end{center}
\end{figure}

Here we present our main results.

\begin{thm}\label{thm:equal90}
Let $\Gamma_0$ and $\gamma$ be defined by (\ref{eqn:initialcurveg}) and (\ref{eqn:contactangle}).

Assume $\gamma=\pi/2$, $n\geq2$.

(1) If the assumption $(A-)$ holds, then there exists $T>0$ such that the interface evolution $\Gamma(t)$ for (\ref{eq:cur}) with initial curve $\Gamma_0$ is not fattening for $0\leq t<T$.

(2) If the assumption $(A+)$ holds, then the interface evolution $\Gamma(t)$ for (\ref{eq:cur}) with initial data $\Gamma_0$ is fattening.
\end{thm}

\begin{thm}\label{thm:less90}
Let $\Gamma_0$ and $\gamma$ be defined by (\ref{eqn:initialcurveg}) and (\ref{eqn:contactangle}).

Then there exist $\alpha_n\in(0,\pi/2)$ ($n\geq2$) such that $\alpha_n\rightarrow \pi/2$, as $n\rightarrow\infty$ and if $0\leq\gamma<\alpha_n$, then there exists $T_{\gamma}$ such that the interface evolution $\Gamma(t)$ for (\ref{eq:cur}) with initial curve $\Gamma_0$ is not fattening for $0\leq t<T_{\gamma}$.
\end{thm}

\begin{thm}\label{thm:less90in2dim}
Let $\Gamma_0$ and $\gamma$ be defined by (\ref{eqn:initialcurveg}) and (\ref{eqn:contactangle}).

Assume $0\leq\gamma<\pi/2$, for $n=1$.

The interface evolution $\Gamma(t)$ for (\ref{eq:cur}) in $\mathbb{R}^{2}$ with initial data $\Gamma_0$ is fattening.
\end{thm}

The definitions of fattening, non-fattening, outer-evolution, inner-evolution and interface evolution are given in section 2.

Theorem \ref{thm:equal90} can be explained by Figure \ref{fig:exist} and Figure \ref{fig:fattening1}. $\varphi$ in Figure \ref{fig:exist} and Figure \ref{fig:fattening1} is the unique viscosity solution of 
$$
\left\{
\begin{array}{lcl}
\dis{\varphi_t=|\nabla \varphi|\textmd{div}(\frac{\nabla \varphi}{|\nabla \varphi|})+A|\nabla \varphi|}\ \textrm{in}\  \mathbb{R}^{n+1}\times(0,T),\\
\varphi(x,y,0)=a_2(x,y)\ \textrm{in}\  \mathbb{R}^{n+1}\times(0,T),
\end{array}
\right.
$$
where $a_2(x,y)$ satisfies $\Gamma_0=\{(x,y)\in\mathbb{R}^{n+1}\mid a_2(x,y)=0\}$ and $\{(x,y)\in\mathbb{R}^{n+1}\mid a_2(x,y)>0\}$ is bounded. Let $\Gamma(t)=\{(x,y)\in\mathbb{R}^{n+1}\mid \varphi(x,y,t)=0\}$.

\begin{figure}[htbp]
	\begin{center}
            \includegraphics[height=7.0cm]{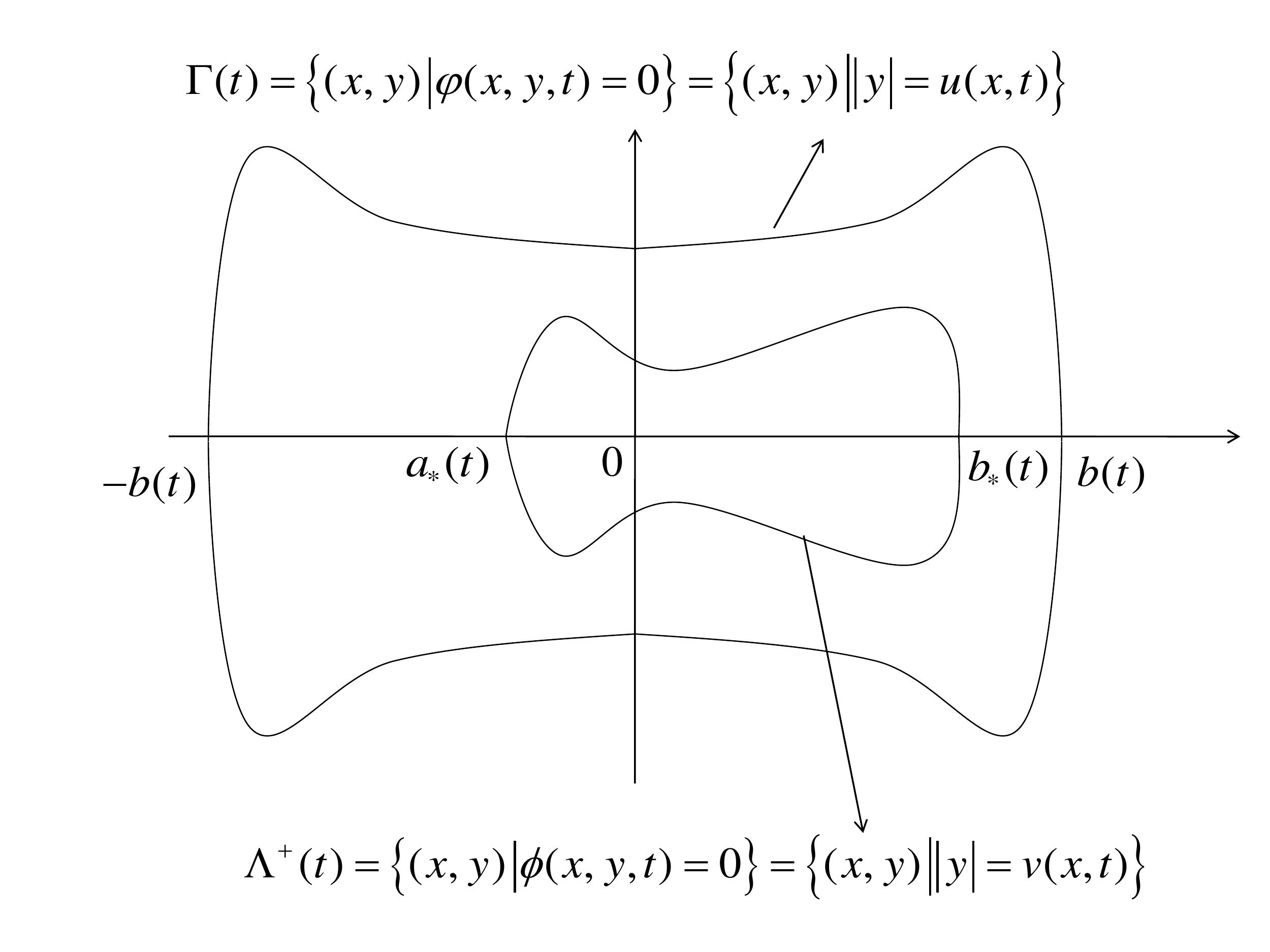}
		\vskip 0pt
		\caption{The case $a_*(t)<0$ in Theorem \ref{thm:equal90}}
        \label{fig:exist}
	\end{center}
\end{figure}

\begin{figure}[htbp]
	\begin{center}
            \includegraphics[height=7.0cm]{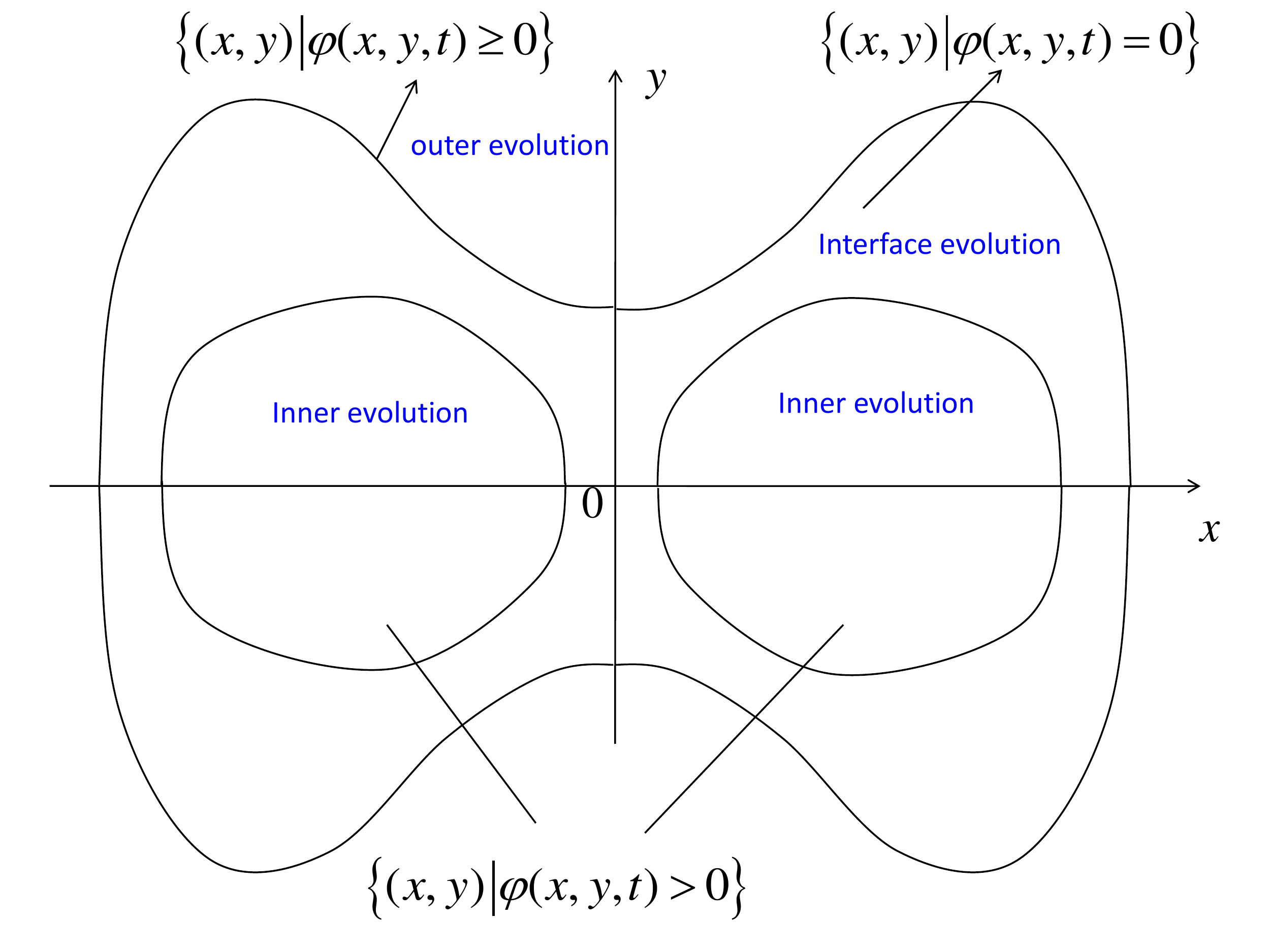}
		\vskip 0pt
		\caption{The case $a_*(t)\geq0$ in Theorem \ref{thm:equal90}}
        \label{fig:fattening1}
	\end{center}
\end{figure}

{\bf Motivation.} Recently, \cite{Z} has considered the mean curvature flow with driving force starting as singular initial curve in the plane. The author gets the same results as in Theorem \ref{thm:equal90} under the condition $n=1$ (see \cite{Z}). In this paper, we give some criteria to judge whether the interface evolution starting as singular hypersurface is fattening or non-fattening in higher dimension. Combining the results in \cite{Z}, we can conclude the results as the following tables. In the following tables, ``Connected'' means that the evolution is a connected set and ``Separated'' means that the evolution consists of two disjoint components.

\begin{table}[h]  
\caption{Singular angle $\gamma=\pi/2$} 
\begin{tabular}{p{3.5cm}|p{3.5cm}|p{3.5cm}}   
\hline  
\hline  

Assumption $(A+)$     & $n=1$ & $n\geq2$ \\  
 
\hline   

Outer evolution & Connected & Connected \\ 

\hline  

Inner evolution & Separated & Separated \\  

\hline  

Result & Fattening & Fattening \\  

\hline  
\hline

\end{tabular}  
\end{table}  

\begin{table}[h]  
\caption{Singular angle $\gamma=\pi/2$} 
\begin{tabular}{p{3.5cm}|p{3.5cm}|p{3.5cm}}   
\hline  
\hline  

Assumption $(A-)$     & $n=1$ & $n\geq2$ \\  
 
\hline   

Outer evolution & Connected & Connected \\ 

\hline  

Inner evolution & Connected & Connected \\  

\hline  

Result & Non-fattening & Non-fattening \\  

\hline  
\hline  
\end{tabular}  
\end{table}  

\begin{table}[h]  
\caption{Singular angle $\gamma<\pi/2$} 
\begin{tabular}{p{3.5cm}|p{3.5cm}|p{3.5cm}}   
\hline  
\hline  

    & $n=1$, $0\leq \gamma<\pi/2$ & $n\geq2$, $0\leq \gamma<\alpha_n$ \\  
 
\hline   

Outer evolution & Connected & Separated \\ 

\hline  

Inner evolution & Separated & Separated \\  

\hline  

Result & Fattening & Non-fattening \\  

\hline  
\hline

\end{tabular}  
\end{table}

{\bf The role of $a_*(t)$ in Theorem \ref{thm:equal90}.} In Section 5, we can prove that there exists a unique solution $(u,b)$  of the following free boundary problem 
\begin{equation}\label{eq:1graph}
u_t=\frac{u_{xx}}{1+u_x^2}-\frac{n-1}{u}+A\sqrt{1+u_x^2},\ -b(t)<x<b(t),\ 0<t<T_1,
\end{equation}
\begin{equation}\label{eq:1bounday}
u(b(t),t)=u(-b(t),t)=0,\ u_x(-b(t),t)=-u_x(b(t),t)=\infty,\ 0<t<T_1,
\end{equation}
\begin{equation}\label{eq:1initial}
u(x,0)=u_0(x),\ -b_0\leq x\leq b_0,
\end{equation}
\begin{equation}\label{eq:innerpositive}
u(x,t)>0,\ -b(t)< x< b(t),\ 0<t<T_1.
\end{equation}
Precisely, we say $(u,b)$ is the solution of (\ref{eq:1graph}), (\ref{eq:1bounday}), (\ref{eq:1initial}) and (\ref{eq:innerpositive}), if

(1) $b(t)$ is a positive function and $b\in C([0,T_1))\cap C^1((0,T_1))$.

(2) $u\in C(\overline{D}_{T_1})\cap C^{2,1}(D_{T_1})$, where $\overline{D}_{T_1}=\cup_{0\leq t<T_1}\big([-b(t),b(t)]\times\{t\}\big)$ and $D_{T_1}=\cup_{0<t<T_1}\big((-b(t),b(t))\times\{t\}\big)$ (We must note that $\overline{D}_{T_1}\neq\overline{D_{T_1}}$).

(3) $(u,b)$ satisfies (\ref{eq:1graph}), (\ref{eq:1bounday}), (\ref{eq:1initial}) and (\ref{eq:innerpositive}).

 Obviously, the flow $\Gamma^*(t)=\{(x,y)\mid |y|=u(x,t),\ -b(t)\leq x\leq b(t)\}$ satisfies (\ref{eq:cur}), (\ref{eq:initial1}) naturally.

Let $(v,a_*,b_*)$ be the solution of the problem (*). If the assumption $(A+)$ holds, the flow
$$
\Lambda^+(t)=\{(x,y)\mid|y|=v(x,t),\ a_*(t)\leq x\leq b_*(t)\}
$$ 
does not intersect the flow
$$
\Lambda^-(t)=\{(-x,y)\mid (x,y)\in \Lambda^+(t)\},
$$  
for $0<t<\delta$. Denote $\Lambda(t)=\Lambda^+(t)\cup\Lambda^-(t)$. Obviously, $\Lambda(t)$ also satisfies (\ref{eq:cur}). Seeing $\Gamma^*(0)=\Lambda(0)=\Gamma_0$, this means that there exist two types of flows $\Gamma^*(t)$ and $\Lambda(t)$ evolving by $V=-\kappa+A$ with the same initial curve $\Gamma_0$. Therefore under this condition, the solution of the original problem (\ref{eq:cur}), (\ref{eq:initial1}) is not unique. Indeed, from the proof of Theorem \ref{thm:equal90}, we see that the flow $\Gamma^*(t)$ is the boundary of the closed evolution and the flow $\Lambda(t)$ is the boundary of open evolution.

If $a_*(t)<0$ for $0<t<\delta$, $\Lambda^+(t)\cap\Lambda^-(t)\neq \emptyset$. Obviously, $\Lambda(t)=\Lambda^+(t)\cup\Lambda^-(t)$ does not satisfy (\ref{eq:cur}). But $\Lambda^+(t)$ plays the role of a sub-solution (in the proof of Lemma \ref{lem:closebou}). Using this sub-solution, the boundaries of the open evolution and closed evolution are away from the $x$-axis. Moreover, it can be proved that the derivatives and the second fundamental forms of them are uniformly bounded. By the uniqueness result  (Proposition \ref{pro:uniq2}), we can prove they are coincide. 

For classical mean curvature flow i.e. $A=0$ and under the condition $\gamma=\pi/2$, since $a_*(t)\geq0$ always holds, the interface evolution is fattening.  

{\bf Background.} 
In 1995, \cite{AAG} considered the classical mean curvature flow in dimension $n$, $n\geq2$. They proved that the singular formations for axisymmetric flow can only be shrinking or pinching. Moreover, they used level set method to show that after pinching, the interface evolution is non-fattening and separated into some disjoint connected components. Indeed, this result can be seen as a special condition $A=0$ and $\gamma=0$ in this paper. 
  
Mean curvature with driving force under the condition $\gamma=\pi/2$ and $n=1$, the curve in plane, has been considered in \cite{Z} recently. The same results as in Theorem \ref{thm:equal90} are given in \cite{Z}. In this paper, we give more general criteria to judge whether the interface evolution starting as singular initial hypersurface is fattening or non-fattening. 

{\bf A short review for mean curvature flow.} For the classical mean curvature flow: $A=0$ in (\ref{eq:cur}), there are many results. Concerning this problem, Huisken \cite{H} showed that any solution that starts out as a compact, smooth and convex surface preserves so until it shrinks to a "round point", its asymptotic shape is a sphere just before it disappears. He prove this result for hypersurfaces in $\mathbb{R}^{n+1}$ with $n\geq2$, while Gage and Hamilton \cite{GH} showed that it still holds when $n=1$, the curves in the plane. Gage and Hamilton also showed that an embedded curve remains embedded, i.e. the curve will not intersect itself. Grayson \cite{Gr} proved the remarkable fact that such family must become convex eventually. Thus, any embedded curve in the plane will shrink to "round point" under mean curvature flow. However in higher dimensions this is not true. Grayson \cite{Gr2} also showed that there exist smooth solutions that become singular before they shrink to a point. His example consists of a barbell: two spherical surfaces connected by a sufficiently thin "neck". In this example, the inward curvature of the neck is so large that it will force the neck to pinch before shrinking. In 1995, A. Altschuler, S. B. Angenent and Y. Giga \cite{AAG} studied the global-in-time solutions whose initial value is a compact, smooth, rotationally symmetric hypersurface given by rotating a graph around an axis by level set method. They proved the hypersurface will separate into two smooth hypersurfaces after pinching. Indeed, the condition in \cite{AAG} is the condition $A=0$ and $\gamma=0$ in this paper.

{\bf Main method.}  Since the initial hypersurface is singular at the origin, we use the level set method to study the problem. The level set method will be introduced in Section 2. For the level set method, \cite{CGG} gives the existence and uniqueness for the viscosity solution of the level set equation. However, in spite of the development of the level set method, it is still difficult to determine whether fattening occurs or not.

Another important tool is the intersection number principle. It was also used in \cite{AAG} that the intersection number between two families evolving by mean curvature flow is non-increase. But for the problem with driving force this number may increase in time. The intersection number between two flows evolving by $V=-\kappa+A$ has been investigated by \cite{Z}. We will give the references in Section 3.

The rest of this paper is organized as follows. In Section 2, we refer to the level set method by using viscosity solution which is established by \cite{CGG}. We also refer to the definitions of open evolution, closed evolution and the basic knowledge concluding comparison principle, monotone convergence theorem and so on. In Section 3, we give some preliminary results including interior gradient estimate, modifying intersection number principle and its application. In Section 4, we prove the main results. In Section 5, we give an example for the phenomenon ``second fattening''. In Section 6, we give the local existence and uniqueness of the problem (\ref{eq:cur}) with some initial hypersurfaces.


\section{Viscosity solution and level set method}

Since the initial curve $\Gamma_0$ has singularity at $(0,0)$, the equation $V=-\kappa+A$ does not make sense at $t=0$. Therefore, we apply the level set method to our problem. In this section, we refer to the level set method in $\mathbb{R}^N$. Let $\Gamma(t)$ be a smooth family of smooth, closed, compact hypersurfaces in $\mathbb{R}^{N}$ given by $\Gamma(t)=\{x|\psi(x,t)=0,x\in\mathbb{R}^N\}$ for some $\psi$ and $\{x\mid \psi(x,t)>0\}$ is bounded. If $\Gamma(t)$ evolves by (\ref{eq:cur}), we can see that $\psi(x,t)$ satisfies 
$$
\dis{\psi_t=|\nabla \psi |\textmd{div}(\frac{\nabla \psi}{|\nabla \psi|})+A|\nabla \psi|\ \textrm{on}\  \{(x,t)\mid\psi(x,t)=0\}}.
$$
Next we consider the equation in whole space
\begin{equation}\label{eq:level}
\dis{\psi_t=|\nabla \psi |\textmd{div}(\frac{\nabla \psi}{|\nabla \psi|})+A|\nabla \psi|\ \textrm{in}\  \mathbb{R}^N\times(0,T)}.
\end{equation}
Equation (\ref{eq:level}) is called the level set equation of (\ref{eq:cur}). Theorem 4.3.1 in \cite{G} gives the existence and uniqueness of the viscosity solution for (\ref{eq:level}) with $\psi(x,0)=\psi_0(x)$. Here $\psi_0(x)$ is a bounded and uniform continuous function.
\begin{figure}[htbp]
	\begin{center}
            \includegraphics[height=5cm]{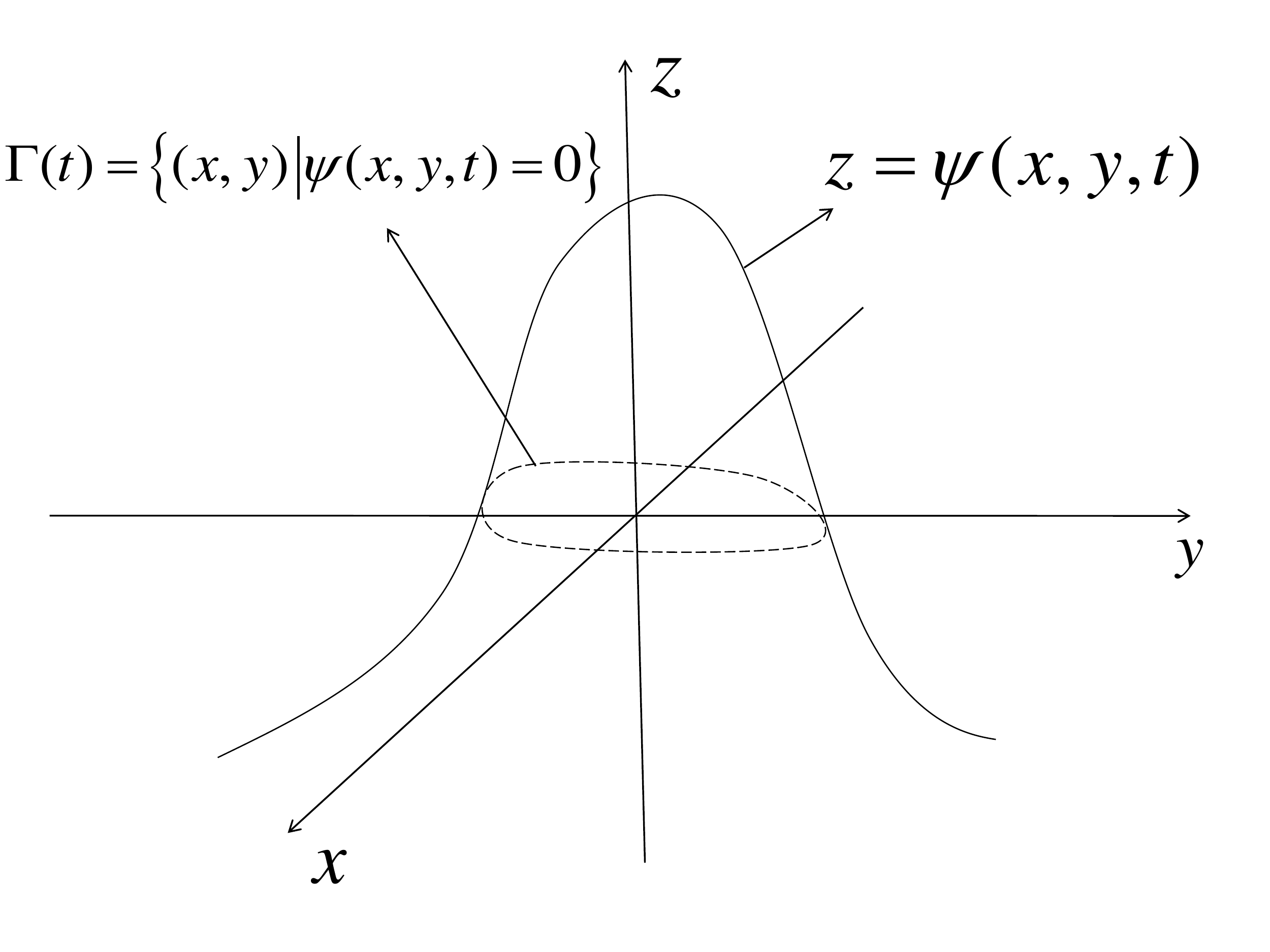}
		\vskip 0pt
		\caption{Level set method in $\mathbb{R}^2$}
        \label{fig:Levelsetmethod}
	\end{center}
\end{figure}

{\bf Level set method.} Using the solution of level set equation, we introduce the level set method. 
\begin{defn}\label{def:evo} (1) Let $D_0$ be a bounded open set in $\mathbb{R}^N$. A family of  open sets $\{D(t)\}_{0<t<T}$ in $\mathbb{R}^N$ is called an $(generlized)$ ``$open$'' or ``$inner$'' $evolution$ of (\ref{eq:cur}) with initial data $D_0$ if there exists a viscosity solution $\psi$ of (\ref{eq:level}) that satisfies
$$
D(t)=\{x\in\mathbb{R}^N \mid \psi(x,t)>0\},\ D_0=\{x\in\mathbb{R}^N\mid \psi(x,0)>0\}.
$$

(2) Let $E_0$ be a bounded closed set in $\mathbb{R}^N$. A family of closed sets $\{E(t)\}_{0<t<T}$ in $\mathbb{R}^N$ is called a $(generlized)$ ``$closed$'' or ``$outer$'' $evolution$ of (\ref{eq:cur}) with initial data $E_0$ if there exists a viscosity solution $\psi$ of  (\ref{eq:level}) that satisfies
$$
E(t)=\{x\in\mathbb{R}^N \mid \psi(x,t)\geq0\},\ E_0=\{x\in\mathbb{R}^N\mid \psi(x,0)\geq0\}.
$$

The set $\Gamma(t)=E(t)\setminus D(t)$ is called an \textit{(generalized) interface evolution} of  (\ref{eq:cur}) with initial data $\Gamma_0=E_0\setminus D_0$.
\end{defn}

\begin{rem}\label{rem:lev} (1) For an open set $D_0$ and $E_0=\overline{D_0}$, we often choose
$$
\psi(x,0)=\max\{\textrm{sd}(x,\partial D_0),-1\}
$$
where
$$
\textrm{sd}(x,\partial D_0)=\left\{
\begin{array}{lcl}
\textrm{dist}(x,\partial D_0), \ x\in D_0,\\
-\textrm{dist}(x,\partial D_0), \ x\notin D_0.
\end{array}\right.
$$

(2) Although the choice of $\psi(x,0)$ is not unique, Theorem 4.2.8 in \cite{G} implies that both the open evolution $D(t)$ and closed evolution $E(t)$ are independent of the choice of $\psi(x,0)$.

(3) In general, even if $E_0=\overline{D_0}$, we can not guarantee $E(t)=\overline{D(t)}$. If $E(t)\setminus D(t)$ has interior points for some $t$, we call the interface evolution is fattening. Respectively, if $E(t)=\overline{D(t)}$ for all $0<t<T$, we say the interface evolution is not fattening. Therefore, for proving non-fattening, it is sufficient to show $\partial D(t)=\partial E(t)$. For proving fattening, it is sufficient to show that there exists a ball $B$ such that $B\subset E(t)\setminus D(t)$.

(4) If $D_0$ and $\overline{D_0}$ are symmetric to $x_i$-axis, then it is also true for $D(t)$ and $E(t)$. Since level set equation (\ref{eq:level}) is invariance under orthogonal transformation.
\end{rem}

We now list some fundamental properties of open evolution and closed evolution of  (\ref{eq:cur}). (All the results listed below can be found in Chapter 4 of \cite{G})

\begin{thm}\label{thm:semi}
(Semigroups \cite{G}). Let $N(t)$ and $M(t)$ be the operators such that
$N(t)D_0=D(t)$ and $M(t)E_0=E(t)$ for $t>0$, where \{D(t)\}, \{E(t)\} are the open evolution and closed evolution with initial data $D_0$ and $E_0$,respectively. Then we have $N(t)D(s)=D(t+s)$ and $M(t)E(s)=E(t+s)$, for any $t>0$, $s>0$. 
\end{thm}

\begin{thm}\label{thm:order}(Order preserving property \cite{G}). Let $D_0$, $D_0^{\prime}$ be two open sets in $\mathbb{R}^N$ and let $E_0$, $E_0^{\prime}$ be two closed sets in $\mathbb{R}^N$ and let $N(t)$, $M(t)$ be the operators defined in Theorem \ref{thm:semi}. Then

(1) $N(t)D_0\subset U(t)D_0^{\prime}$, if $D_0\subset D_0^{\prime}$;

(2) $M(t)E_0\subset M(t)E_0^{\prime}$, if $E_0\subset E_0^{\prime}$;

(3) $N(t)D_0\subset M(t)E_0^{\prime}$, if $D_0\subset E_0^{\prime}$;

(4) $E_0\subset D_0$ and $\textrm{dist}(E_0,\partial D_0)>0$, then $M(t)E_0\subset N(t)D_0$.
\end{thm}

\begin{thm}\label{thm:mon}(Monotone convergence \cite{G}).

(1) Let $D(t)$ and $\{D_j(t)\}$ be open evolutions with initial data $D_0$ and $D_{j0}$ respectively. If $D_{j0}\uparrow D_0$, then $D_j(t)\uparrow D(t)$ for $t>0$, i.e., $\bigcup\limits_{j\geq1}D_j(t)=D(t)$;

(2) Let $E(t)$ and $\{E_j(t)\}$ be closed evolutions with initial data $E_0$ and $E_{j0}$ respectively. If $E_{j0}\downarrow E_0$, then $E_j(t)\downarrow E(t)$ for $t>0$, i.e., $\bigcap\limits_{j\geq1}E_j(t)=E(t)$.
\end{thm}

\begin{rem}\label{rem:ori} Let $\{D_1(t)\}$ and $\{D_2(t)\}$ be the open evolution of $V=-\kappa+A$.  For $A>0$, even if $D_1(0)$ and $D_2(0)$ are disjoint, $D_1(t)$ and $D_2(t)$ may have intersections. The basic reason is that the level set equation (\ref{eq:level}) is not orientation free (It is not true that $-u$ is also a solution for (\ref{eq:level}) when $u$ is a solution for (\ref{eq:level})). 
\end{rem}

\begin{thm}\label{thm:sep}(Separate)
Let $\{D_1(t)\}_{0\leq t<T}$ be the open evolution of $V=-\kappa+A$ and $\{D_2(t)\}_{0\leq t<T}$ be the open evolution of $V=-\kappa-A$. If $D_1(0)\cap D_2(0)=\emptyset$, then $D_1(t)\cap D_2(t)=\emptyset$ for $0\leq t<T$. 
\end{thm}

The proof of this theorem is similar to that of Theorem 3.5 in \cite{AAG}. We omit it.

 In order to prove the fattening result, we need the following lemma. 

\begin{lem}\label{lem:sep}(Lemma 2.8 in \cite{Z}) Assume that $D_1(t)$ and $D_2(t)$ are the open evolution of (\ref{eq:level}) with $D_1(0)=U_1$ and $D_2(0)=U_2$, respectively. And $D(t)$ is denoted as the open evolution of (\ref{eq:level}) with $D(0)=U_1\cup U_2$. If $D_1(t)\cap D_2(t)=\emptyset$ for $0\leq t\leq T$, then  $D(t)=D_1(t)\cup D_2(t)$ for $0\leq t\leq T$.
\end{lem}

\begin{thm}\label{thm:openevolutionmeancurvature}(Local smoothness for graphs) Suppose that $\psi$ is a viscosity solution of (\ref{eq:level}). Assume in an open region $U\times(t_1,t_2)$,
$$
\{(x,t)\mid \psi=0\}\cap U=\{(x,t)\mid x_N=g(x^{\prime},t), x^{\prime}\in U^{\prime}\},
$$
where $x^{\prime}=(x_1,\cdots,x_{N-1})$, $U^{\prime}=U\cap\{x_N=0\}$ and g is continuous in $U^{\prime}\times(t_1,t_2)$. Then the function $g$ is a viscosity solution of 
$$
g_t=\left(\delta_{ij}-\frac{g_{x_i}g_{x_j}}{1+|\nabla g|^2}\right)g_{x_ix_j}+ A\sqrt{1+|\nabla g|^2},
$$
or 
$$
g_t=\left(\delta_{ij}-\frac{g_{x_i}g_{x_j}}{1+|\nabla g|^2}\right)g_{x_ix_j}-A\sqrt{1+|\nabla g|^2}.
$$
If the direction of the driving force of $\{(x,t)\mid \psi=0\}\cap U$ is upward (downward), we choose ``$+$'' (``$-$'') in above graph equation.

Moreover, $g$ is $C^{\infty}$ in the region $U^{\prime}\times(t_1,t_2)$.
\end{thm}
The proof is similar to that in \cite{ES} (Theorem 5.1 and Theorem 5.4), and we omit it.

\section{Preliminaries}

In this section, we prepare some preliminaries for future purposes. The results in this section can be found in \cite{Z}.

{\bf Graph equation.} Let $u(x,t)$ be a function on an open subset of $\mathbb{R}^n\times \mathbb{R}$, then the graph of $u(x,t)$ is a family of hypersurfaces in $\mathbb{R}^{n+1}$. The family of hypersurfaces moves by $V=-\kappa+A$ if and only if 
$$
u_t=\left(\delta_{ij}-\frac{u_{x_i}u_{x_j}}{1+|\nabla u|^2}\right)u_{x_ix_j}\pm A\sqrt{1+|\nabla u|^2},
$$
where the signs of the last terms are determined by the direction of the driving force.

{\bf Horizontal and vertical graph equation.} If $\Gamma(t)$ is a family of rotationally symmetric hypersurfaces in $\mathbb{R}^{n+1}$, then for each $t>0$, a part of $\Gamma(t)$ may be represented either as a horizontal graph, $r=u(x,t)$, or a vertical graph, $x=v(r,t)$, where $r=\sqrt{y_1^2+y_2^2+\cdots+y_n^2}$, $(x,y)\in \mathbb{R}\times\mathbb{R}^n$.

If $\Gamma(t)$ is given as a horizontal graph, then $\Gamma(t)$ evolves by $V=-\kappa+A$ in $\mathbb{R}^{n+1}$ and the direction of the driving force points to the positive direction of $r=|y|$ axis if and only if $u$ satisfies the horizontal graph equation 
\begin{equation}\label{eq:1horizontal}
\frac{\partial u}{\partial t}=\frac{u_{xx}}{1+u_x^2}-\frac{n-1}{u}+A\sqrt{1+u_x^2}.
\end{equation}
If $\Gamma(t)$ is given as a vertical graph, then $\Gamma(t)$ evolves by $V=-\kappa+A$ in $\mathbb{R}^{n+1}$ if and only if $v$ satisfies the vertical graph equation
\begin{equation}\label{eq:1vertical+}
\frac{\partial v}{\partial t}=\frac{v_{rr}}{1+v_r^2}+\frac{n-1}{r}v_r+ A\sqrt{1+v_r^2},
\end{equation}
or
\begin{equation}\label{eq:1vertical-}
\frac{\partial v}{\partial t}=\frac{v_{rr}}{1+v_r^2}+\frac{n-1}{r}v_r-A\sqrt{1+v_r^2},
\end{equation}
where the sign of the last term is determined by the direction of the driving force (We choose ``$+$($-$)'' when the direction of the driving force is the positive (negative) direction of $x$ axis).

\begin{thm}\label{thm:es}(Theorem 3.1 in \cite{Z}) For $u\in C^3(\Omega_{T})\cap C^0 (\overline{\Omega}_T)$, $u$ satisfies
\begin{equation}\label{eq:graph}
u_t=\dis{\left(\delta_{ij}-\frac{u_{x_i}u_{x_j}}{1+|\nabla u|^2}\right)u_{x_ix_j}\pm A\sqrt{1+|\nabla u|^2}},
\end{equation}
For the condition ``$+$''(``$-$''), we assume $u<0$ ($u>0$) in $\Omega_T$, $u(0,T)=-v_0$ ($u(0,T)=v_0$). Then
$$
|\nabla u(0,T)|\leq (3+16v_0)e^{2K},
$$
where $\dis{K=20v_0^2(4n+\frac{1}{T}+4A+\frac{A}{2v_0})}+2$, $\Omega_T=B_1(0)\times (0, 2T)\subset \mathbb{R}^n\times(0,\infty)$ and
$$
\delta_{ij}=\left\{
\begin{array}{lcl}
1,\ i=j,\\
0,\ i\neq j.
\end{array}
\right.
$$
\end{thm}

\begin{rem}\label{rem:es} (1) In Theorem \ref{thm:es}, $\Omega_T$ can be replaced by $\Omega_T=B_R(x_0)\times(0,2T)$ and $v_0=u(x_0,T)$. Then the conclusion becomes
$$
|\nabla u(x_0,T)|\leq e^{2K}(3+16\frac{v_0}{R}),
$$
where $\dis{K=20\frac{v_0^2}{R^2}\left(4n+\frac{R^2}{T}+\frac{4A}{R}+\frac{A}{2v_0}\right)}+2$. Here We set $v(x,t)=\dis{\frac{u(Rx+x_0,R^2t)}{R}}$, and apply Theorem \ref{thm:es} to $v(x,t)$.

(2) When $u$ is the solution of (\ref{eq:graph}) for ``+'' without the assumption ``$u<0$'', if we set
$$v=u-M-\epsilon$$
where $M=\sup\limits_{\overline{\Omega}_T} |u|$ and $\epsilon>0$. and apply (1) in Remark \ref{rem:es} to $v$, we get
$$|\nabla u(0,T)|\leq\left(3+16\frac{M-u(0,T)+\epsilon}{R}\right)e^{2\widetilde{K}_{\epsilon}},$$
where $\dis{\widetilde{K}_{\epsilon}=\frac{20(M-u(0,T)+\epsilon)^2}{R^2}\left(4n+\frac{R^2}{T}+\frac{4A}{R}
+\frac{A}{2(M+\epsilon-u(0,T))}\right)}$+2.

As $\epsilon\rightarrow0$, we have
$$
|\nabla u(0,T)|\leq \left(3+32\frac{M}{R}\right)e^{2\widetilde{K}},
$$
where $\dis{\widetilde{K}=\frac{80M^2}{R^2}\left(4n+\frac{R^2}{T}+\frac{4A}{R}\right)
+\frac{20AM}{R^2}}+2$.
\end{rem}

Then we use the (2) in Remark \ref{rem:es} and the same method as in \cite{AAG} to prove the next corollary.

\begin{cor}\label{cor:es} For $s_1<s_2$, $\rho>0$ and $x_0\in\mathbb{R}^n$ we set
$$
\Omega=B_{\rho}(x_0)\times(s_1,s_2).
$$
Suppose that $u\in C^3(\Omega)$ solves the equation (\ref{eq:graph}) in $\Omega$ with $M=\sup\limits_{\overline{\Omega}}|u|<\infty$. For any $\epsilon>0$ there is a constant $C=C(M,\epsilon,n)$ such that
$$|\nabla u|\leq C \ \textrm{on}\ \Omega_{\epsilon}=B_{\rho-\epsilon}(x_0)\times(s_1+\epsilon^2,s_2).$$
\end{cor}
\begin{rem}\label{rem:hes}
(1) From Corollary \ref{cor:es} and \cite{LSU}, there exist $C_k(M,\epsilon,n)$ such that 
$$
|\nabla^k u|\leq C_k,\ (x,t)\in B_{\rho-2\epsilon}(x_0)\times(s_1+2\epsilon^2,s_2).
$$
(2) Noting that $C$ and $C_k$ are all independent of $s_2$, if the solution $u$ exists for all $t>s_1$, then $s_2$ can be taken as $\infty$ in Corollary \ref{cor:es}.
\end{rem}

\section{Intersection number principle}

Next, we begin to introduce the modifying intersection number principle.

{\bf Intersection number for rotationally symmetric hypersurfaces.} Let $r=|y|$, $y\in\mathbb{R}^n$. For two rotationally symmetric hypersurfaces $\Gamma_1(t)$ and $\Gamma_2(t)$, given by $\Gamma_1(t)=\{(x,y)\in \mathbb{R}\times\mathbb{R}^{n}\mid r=u_1(x,t)\}$ and $\Gamma_2(t)=\{(x,y)\in \mathbb{R}\times\mathbb{R}^{n}\mid r=u_2(x,t)\}$, define the number of intersections between $u_1(\cdot,t)$ and $u_2(\cdot,t)$ as the intersection number between $\Gamma_1(t)$ and $\Gamma_2(t)$ denoted by $\mathcal{Z}[\Gamma_1(t),\Gamma_2(t)]$. 

It is well known that the intersection number between two families of rotationally symmetric hypersurfaces $\Gamma_1(t)$ and $\Gamma_2(t)$ evolving by $V=-\kappa$ is not increasing (\cite{A2}). But it is not true for $V=-\kappa+A$. To conquer this difficulty, in \cite{Z}, they give some results for the intersection number in the case $V=-\kappa+A$. 

\begin{thm}\label{thm:intersection1}(Remark 4.4 in \cite{Z})
Let $\Gamma_1(t)$ and $\Gamma_2(t)$ be two closed, compact, rotationally symmetric hypersurfaces given by $\Gamma_1(t)=\{(x,y)\in \mathbb{R}\times\mathbb{R}^n\mid r=u_1(x,t),a_1(t)\leq x\leq b_1(t)\}$,\ $\Gamma_2(t)=\{(x,y)\in \mathbb{R}\times\mathbb{R}^n\mid r=u_2(x,t),\ a_2(t)\leq x\leq b_2(t)\}$. If $\Gamma_i(t)$ evolve by $V=-\kappa+A$ in $\mathbb{R}^{n+1}$, $i=1,2$. Then 
\\
(a). $\mathcal{Z}(\Gamma_1(t),\Gamma_2(t))>0$ does not increase in $t$ when $t$ satisfies $\mathcal{Z}(\Gamma_1(t),\Gamma_2(t))>0$.
\\
(b). If $\mathcal{Z}(\Gamma_1(t_0),\Gamma_2(t_0))=0$, then $\mathcal{Z}(\Gamma_1(t),\Gamma_2(t))\leq 1$ for $t_0<t<T$. 
\end{thm}

Using Theorem \ref{thm:intersection1}, we can prove following Theorem \ref{thm:grad} and Theorem \ref{thm:gu}.

\begin{thm}\label{thm:grad}(Theorem 4.2 in \cite{Z}) $\Gamma(t)=\{(x,y)\in\mathbb{R}^{n+1}\mid r=u(x,t),a_2(t)\leq x\leq b_2(t)\}$ is a smooth family of closed, smooth hypersurfaces in $\mathbb{R}^{n+1}$ for $0<t<T$. If $\Gamma(t)$ evolves by $V=-\kappa+A$ in $\mathbb{R}^{n+1}$,  there is a function $\sigma$: $\mathbb{R}_+\times\mathbb{R}_+\rightarrow\mathbb{R}$ such that
$$
|u_x(x,t)|\leq \sigma(t,u(x,t))
$$
holds for $0<t<T$, $a_2(t)<x<b_2(t)$. The function $\sigma$ depends only on $M=\max\limits_{a_2(0)<x<b_2(0)} u(x,0)$ and $T$.
\end{thm}

\begin{thm}\label{thm:gu}(Theorem 4.5 in \cite{Z}) Let $\Gamma(t)$, $t\in [0,T)$, be a family of smooth hypersurfaces evolving by $V=-\kappa+A$ in $\mathbb{R}^{n+1}$. If $\Gamma(0)$ is obtained by rotating the graph of a function around the $x$-axis, $\Gamma(t)$ also has the same symmetry as $\Gamma(0)$ for $t\in[0,T).$
\end{thm}

In our problem, the hypersurface evolving by $V=-\kappa+A$ maybe intersect itself at $x$-axis. To conquer this difficulty, we refer to the definition of $\alpha$-domain in \cite{AAG}.

\begin{defn}\label{def:alphad} We say a domain $U$ is an $\alpha$-domain if

(1). $U\subset \mathbb{R}^{n+1}$ is an open set of the form
$$
U=\{(x,y)\in\mathbb{R}\times\mathbb{R}^n\mid r<u(x)\}.
$$

(2). $I=\{x\in\mathbb{R}\mid u(x)>0 \}$ is a bounded, connected interval. Let the endpoints of $I$ be $a_1<a_2$.

(3). $u$ is smooth on $I$;

(4). $\partial U$ intersects each cylinder $\partial C_{\rho}$ with $0<\rho\leq\alpha$ twice and these intersections are transverse, where $C_{\rho}=\{(x,y)\in\mathbb{R}^{n+1}\mid r<\rho\}$.
\end{defn}

\begin{figure}[htbp]
	\begin{center}
            \includegraphics[height=5cm]{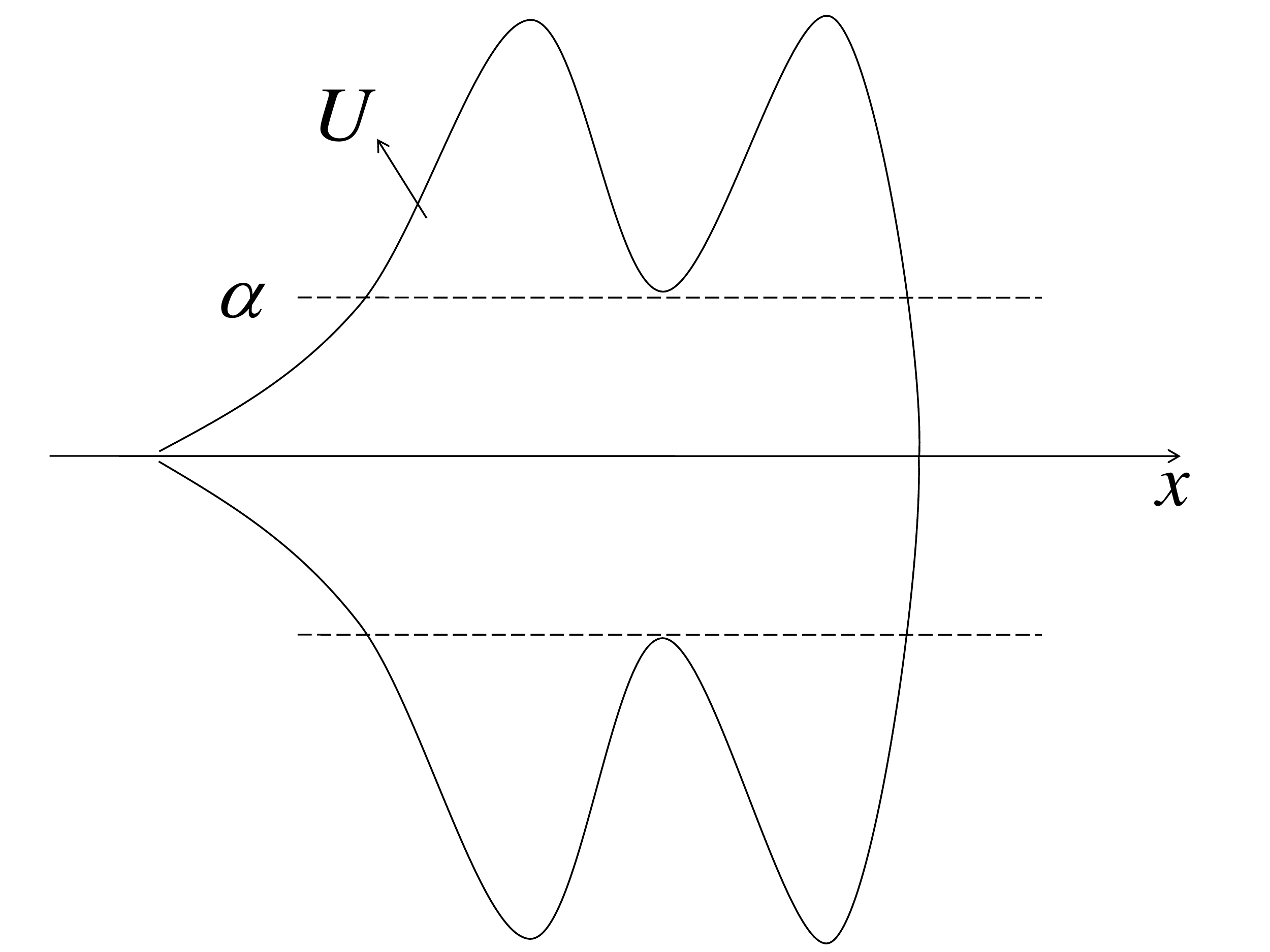}
		\vskip 0pt
		\caption{$\alpha$-domain}
        \label{fig:alphadom}
	\end{center}
\end{figure}
We observe that the boundary $\partial U$ of an $\alpha$-domain $U$ does not intersect itself at $y=0$. The condition (3) implies $\partial U$ is a smooth curve, except possibly at its endpoints $(a_1,0),(a_2,0)$. The condition (4) implies that there exist $\delta_1,\delta_2>0$ such that
$$
u(a_1+\delta_1)=u(a_2-\delta_2)=\alpha,
$$
and
$$
u^{\prime}(x)=\left\{
\begin{array}{lcl}
>0,\ x\in(a_1,a_1+\delta_1],\\
<0,\ x\in[a_2-\delta_2,a_2).
\end{array}
\right.
$$
Therefore, the inverse of $u|_{[a_1,a_1+\delta_1]}$ and $u|_{[a_2-\delta_2,a_2]}$ exist, denoted by $v_1$, $v_2:[0,\alpha]\rightarrow\mathbb{R}$. By the implicit function theorem, they are smooth in $(0,\alpha]$. Moreover, $v_1^{\prime}(r)>0$, $v_2^{\prime}(r)<0$, $(0<r\leq\alpha)$ and
$$
\partial U\cap C_{\alpha}=\{(x,y)\in\mathbb{R}^{n+1}\mid0\leq r\leq\alpha,\ x=v_i(r),\ i=1,2\}.
$$
The two components of $\partial U\cap C_{\alpha}$ are called the left and right caps of $\partial U$.

Since $U$ is an $\alpha$-domain, by Theorem \ref{thm:partialUmeancurvature}, there exists $T_U$ such that $\partial D(t)$ is smooth and $D(t)=\{(x,y)\in\mathbb{R}^2\mid |y|<u(x,t), a(t)< x < b(t)\}$ for $0<t<T_U$. Here $(u,a,b)$ satisfies
\begin{equation}\label{eq:2graph}
u_t=\frac{u_{xx}}{1+u_x^2}-\frac{n-1}{u}+A\sqrt{1+u_x^2},\ a(t)<x<b(t),\ 0<t<T_U,
\end{equation}
\begin{equation}\label{eq:2bounday}
u(b(t),t)=u(a(t),t)=0,\ u_x(a(t),t)=-u_x(b(t),t)=\infty,\ 0<t<T_U,
\end{equation}
\begin{equation}\label{eq:2innerpositive}
u(x,t)>0,\ a(t)< x< b(t),\ 0<t<T_U.
\end{equation} 

\begin{lem}\label{lem:alphadn}
For $n\geq2$, there exists a $t_U^{\alpha}>0$ such that $D(t)$ is an $\alpha(t)$-domain for all $0<t<\min\{t_U^{\alpha},T_U\}$, where $\alpha(t)$ is the solution of the following equation
\begin{equation}\label{eq:cylinder}
\alpha^{\prime}(t)=A-\frac{n-1}{\alpha(t)}
\end{equation}
with initial data $\alpha(0)=\alpha$.
\end{lem}

For the problem,
$$
\left\{
\begin{array}{lcl}
\alpha^{\prime}(t)=A-\dis{\frac{n-1}{\alpha(t)}},\ t>0,\\
\alpha(0)=\alpha,
\end{array}
\right.
$$
the following hold.

(1) when $\alpha<(n-1)/A$, there exists $T_{\alpha}<\infty$ such that $\alpha(t)\downarrow 0$ as $t\rightarrow T_{\alpha}$ and $\lim\limits_{t\rightarrow -\infty}\alpha(t)=(n-1)/A $; 

(2) when $\alpha=(n-1)/A$, $\alpha(t)=(n-1)/A$ for $0\leq t<\infty$;

(3) when $\alpha>(n-1)/A$, $\alpha(t)\uparrow\infty$ as $t\rightarrow\infty$ and $\lim\limits_{t\rightarrow -\infty}\alpha(t)=(n-1)/A $. 

\begin{proof} Since $U$ is not contained in the cylinder $\overline{C}_{\alpha}$, there is a small ball $B_{\epsilon}(P)\subset U\setminus\overline{C}_{\alpha}$. By (1) in Theorem \ref{thm:order}, $D(t)$ contains the ball $B_{\epsilon(t)}(P)$ for $0<t<\delta_1$, where $\epsilon(t)$ is the solution of the following equation
\begin{equation}\label{eq:invercircle}
\epsilon^{\prime}(t)=-A-\frac{n}{\epsilon(t)},\ 0<t<\delta_1
\end{equation}
with initial data $\epsilon(0)=\epsilon$ and $\delta_1$ is the maximal time for the solution existing.

Let $\delta_2$ denote the maximal existence time of the solution for equation (\ref{eq:cylinder}) with initial data $\alpha(0)=\alpha$.

Theorem \ref{thm:sep} implies $B_{\epsilon(t)}(P)\cap\overline{C}_{\alpha(t)}=\emptyset$ for $0<t<t_U^{\alpha}$. Here
$$
t_U^{\alpha}=\min\{\delta_1,\delta_2\}.
$$

Fix $0<\rho<\alpha(t_0)$ and $0<t_0<\min\{t_U^{\alpha},T_U\}$, and let $\rho(t)$ be the solution of the equation (\ref{eq:cylinder}) with initial data $\rho(0)=\rho$. Next we prove that $\partial C_{\rho}$ intersects $\partial D(t_0)$ at most twice.

By comparison principle for ordinary differential equation, $\dis{\rho(-t_0)<\alpha}$ holds. Therefore, $y=\rho(-t_0)$ intersects $y=u(x,0)$ only twice. Noting $\rho(t-t_0)>0$ for $t\geq 0$, 
$$
y=\rho(t-t_0)>u(a(t),t)=0,\ y=\rho(t-t_0)>u(b(t),t)=0\ \text{for}\ t>0.
$$
By Theorem D in \cite{A1}, $\mathcal{Z}_{[a(t),b(t)]}(\rho(t-t_0),u(\cdot,t))$ is not increasing for $t>0$. Here $\mathcal{Z}_I(u_1(\cdot,t),u_2(\cdot,t))$ denotes the intersection number between $u_1$ and $u_2$ in $I$.  

Therefore,  
$$
\mathcal{Z}_{[a(t),b(t)]}(\rho(t-t_0),u(\cdot,t))\leq  \mathcal{Z}_{[a(0),b(0)]}(\rho(-t_0),u(\cdot,0))=2.
$$

This means that $\partial C_{\rho}$ intersects $\partial D(t_0)$ at most twice.

On the other hand, by the choice of $t_{U}^{\alpha}$, $D(t)$ contains the ball $B_{\epsilon(t)}(P)$. Since $B_{\epsilon(t)}(P)$ lies outside of the cylinder $C_{\alpha(t)}$, each $\partial C_{\rho}$ must intersect $\partial D(t)$ at least twice for any $0<t<\min\{t_{U}^{\alpha},T_U\}$.

Therefore $\partial C_{\rho}$ intersects $\partial D(t_0)$ exactly twice. Consequently, $D(t)$ is an $\alpha(t)$-domain for all $0<t<\min\{t_{U}^{\alpha},T_U\}$.
\end{proof}

\begin{prop}\label{pro:sin} For $t_U^{\alpha}$ and $T_U$ given in Lemma \ref{lem:alphadn} holds, $t_U^{\alpha}\leq T_U$.
\end{prop}

To prove this proposition, we need the following lemma.

\begin{lem}\label{lem:sing1} Assume that $D(t)=\{(x,y)\mid |y|<u(x,t), a(t)\leq x\leq b(t)\}$ is a $\rho$-domain for $0<t<T$. Let  $w_1<w_2$ such that
$$
C_\rho\cap\partial D(t)=\{(x,y)\mid x=w_1(y,t)\ or\ x=w_2(y,t)\}.
$$
 Then
$$\lim\limits_{t\rightarrow T}w_1(y,t)=w_1(y,T)\ \ \ \ \ \textrm{and} \ \ \ \ \lim\limits_{t\rightarrow T}w_2(y,t)=w_2(y,T)$$
exist and these convergences are uniformly convergent for $|y|\leq\frac{\rho}{2}$. Moreover, $v_1(r_1,T)<v_1(r_2,T)$ and $v_2(r_1,T)>v_2(r_2,T)$ for $0<r_1<r_2<\frac{\rho}{2}$,
where $v_1(r,t)=w_1(y,t)$ and $v_2(r,t)=w_2(y,t)$.
\end{lem}
\begin{proof} $w_1(y,t)$ and $w_2(y,t)$ satisfy the equation (\ref{eq:graph}), respectively for "$\mp$". We only prove the result for $w_1(y,t)$. Since $w_1$ is uniformly bounded, Corollary \ref{cor:es} and Remark \ref{rem:hes} imply that derivatives $\nabla^j_y w_1$, $j=1,2$, are uniformly bounded for $0\leq|y|\leq\frac{\rho}{2}$, $\frac{T}{2}\leq t<T$. Consequently, $\frac{\partial w_{1}}{\partial t}$ is bounded for $0\leq|y|\leq\frac{\rho}{2}$, $\frac{T}{2}\leq t<T$. So there exists $w_1(y,T)$ such that $w_1(y,t)$ converges to $w_1(y,T)$ uniformly for $0\leq|y|\leq\frac{\rho}{2}$, as $t\rightarrow T$.

Note that the following hold
$$
\frac{\partial v_1}{\partial r}(\frac{\rho}{2},t)>0,\ \frac{\partial v_1}{\partial r}(0,t)=0\ \text{for}\ 0<t<T
$$ 
and 
$$
\frac{\partial v_1}{\partial r}(r,0)>0\ \text{for}\ 0< r<\frac{\rho}{2}.
$$
In fact, differentiating (\ref{eq:1vertical-}) in $r$, 
\begin{equation}\label{eq:deriveeq}
p_t=a(r,t)p_{rr}+b(r,t)p_r+c(r,t)p,
\end{equation}
where $p=w_r$, $a(r,t)=1/(1+w_r^2)$, $b(r,t)=-2w_rw_{rr}/(1+w_r^2)^2+(n-1)/r-Aw_r/\sqrt{1+w_r^2}$, $c(r,t)=-(n-1)/r^2$. Strong maximum principle implies
 $$
\frac{\partial v_1}{\partial r}>0\ \text{for}\  0<r<\frac{\rho}{2},\ 0<t\leq T.
$$ 
Therefore $v_1(r_1,T)<v_1(r_2,T)$ for $0<r_1<r_2<\frac{\rho}{2}$. Similarly, we can prove the conclusion for $v_2$.
\end{proof}

\begin{proof}[Proof of Proposition \ref{pro:sin}.] If $T_U<t_U^{\alpha}$. By Lemma \ref{lem:alphadn}, there exists $\rho>0$ such that $D(t)$ is a $\rho$-domain for $0<t<T_U$. 

We divide $\partial D(t)$ into two parts: $\partial D(t)=(\partial D(t)\cap\{r< \rho/2\})\cup(\partial D(t)\cap\{r\geq \rho/2\})$.
\\
{\bf Step 1.} $\partial D(t)\cap\{r< \rho/2\}$

 Since $\partial D(t)$ is a $\rho$-domain, there exist $w_1<w_2$ such that $\partial D(t)\cap\{r<\rho\}=\{(x,y)\mid x=w_1(y,t), |y|<\rho\}\cup\{(x,y)\mid x=w_2(y,t), |y|<\rho\}$. By the same argument as in Lemma \ref{lem:sing1}, $\nabla^j_y w_i$, $j=1,2$, $i=1,2$, are uniformly bounded for $0\leq|y|\leq\frac{\rho}{2}$, $\frac{T_U}{2}\leq t<T_U$. Therefore, the mean curvature of $\partial D(t)\cap\{r< \rho/2\}$ is uniformly bounded for $\frac{T_U}{2}\leq t<T_U$. 
\\
{\bf Step 2.} $\partial D(t)\cap\{r\geq\rho/2\}$

Recalling $\partial D(t)=\{(x,y)\mid |y|=u(x,t),a(t)\leq x\leq b(t)\}$, by Lemma \ref{lem:sing1}, $v_1(\rho/4,T_U)<v_1(\rho/2,T_U)$ and $v_2(\rho/4,T_U)>v_2(\rho/2,T_U)$. Then for any sufficiently small $\epsilon>0$ for all $t$ close to $T_U$ there holds
\begin{equation}\label{eq:subset1}
[v_1(\rho/2,t),v_2(\rho/2,t)]\subset (v_1(\rho/4,t)+\epsilon,v_2(\rho/4,t)-\epsilon).
\end{equation}
Theorem \ref{thm:grad} shows that $u_x$ is bounded for $|y|\geq\rho/4$. i.e. $u_x$ is bounded in $[v_1(\rho/4,t),v_2(\rho/4,t)]$, $t$ close to $T_U$. Remark \ref{rem:hes} implies that $u_{xx}$ is uniformly bounded in $(v_1(\rho/4,t)+\epsilon,v_2(\rho/4,t)-\epsilon)$, $t$ close to $T_U$. 

Therefore, (\ref{eq:subset1}) shows that $u_x$ and $u_{xx}$ are uniformly bounded for $x\in[v_1(\rho/2,t),v_2(\rho/2,t)]$, $t$ close to $T_U$. Consequently, the curvature of $\partial D(t)\cap\{r\geq\rho/2\}$ is bounded for $t$ close to $T_U$. Here we show that the curvature of $\partial D(t)$ is uniformly bounded as $t\uparrow T_U$. It contradicts that $\partial D(t)$ become singular at $T_U$.
\end{proof}

\begin{rem}\label{rem:time} In Lemma \ref{lem:alphadn}, $0<t<\min\{t_{U}^{\alpha},T_U\}$ is equivalent to $0<t<t_{U}^{\alpha}$. By the choice of $t_{U}^{\alpha}$, if $U\subset W$, $t_{U}^{\alpha}\leq t_{W}^{\alpha}$.
\end{rem}

\section{Proof of Theorem \ref{thm:equal90}}

Denote $U=\{(x,y)\in \mathbb{R}^{n+1}\mid |y|<u_0(x),-b_0\leq x\leq b_0\}$. By assumption of $u_0$ in Section 1, we know that $U\cap\{x\geq0\}$ is an $\alpha$-domain with smooth boundary, for some $\alpha>0$. 

We choose vector field $X\in C^1(\mathbb{R}^{n+1}\setminus\{O\}\rightarrow\mathbb{R}^{n+1})$ such that

(i) At any $P\in \partial U$ not on the $x$-axis has $\langle X,\textbf{n}(P)\rangle>0$, $\textbf{n}(P)$ is the inward unit normal vector to $\partial U$ at $P$.

(ii) We set $X((x,y))=(0,-y/|y|)$, near $O$ and set $X=(-1,0,\cdots,0)$ near $(b,0,\cdots,0)$, $X=(1,0,\cdots,0)$ near $(-b,0,\cdots,0)$. 
\\
We note that $X$ has no definition at $O$.

Since $X\neq0 $ on $\partial U\setminus \{O\}$ and $|X|=1$ near $O$, by continuity, there exists a neighbourhood $V\supset\partial U$ such that $|X|\geq \delta>0$ for some $\delta>0$ in $V\setminus \{(0,0)\}$. 

\begin{prop}\label{pro:sigma2} For $\rho$ small enough, there exists a smooth hypersurface $\Sigma\subset V\setminus\{O\}$ with

(i) $X(P)\notin T_P\Sigma$ at all $P\in\Sigma$, i.e., $\Sigma$ is transverse to the vector field $X$;

(ii) $\Sigma=\partial U$ in $\{(x,y)\mid|y|\geq2\rho\}$;

(iii) $\Sigma\cap\{(x,y)\mid|y|\leq\rho\}$ consists of discs $\Delta_{\pm c}=\{(\pm c,y)\mid|y|\leq\rho\}$ and pipe $B_d=\{(x,y)\mid-d \leq x\leq d,|y|=\rho\}$.
\end{prop}

\begin{figure}[htbp]
	\begin{center}
            \includegraphics[height=5cm]{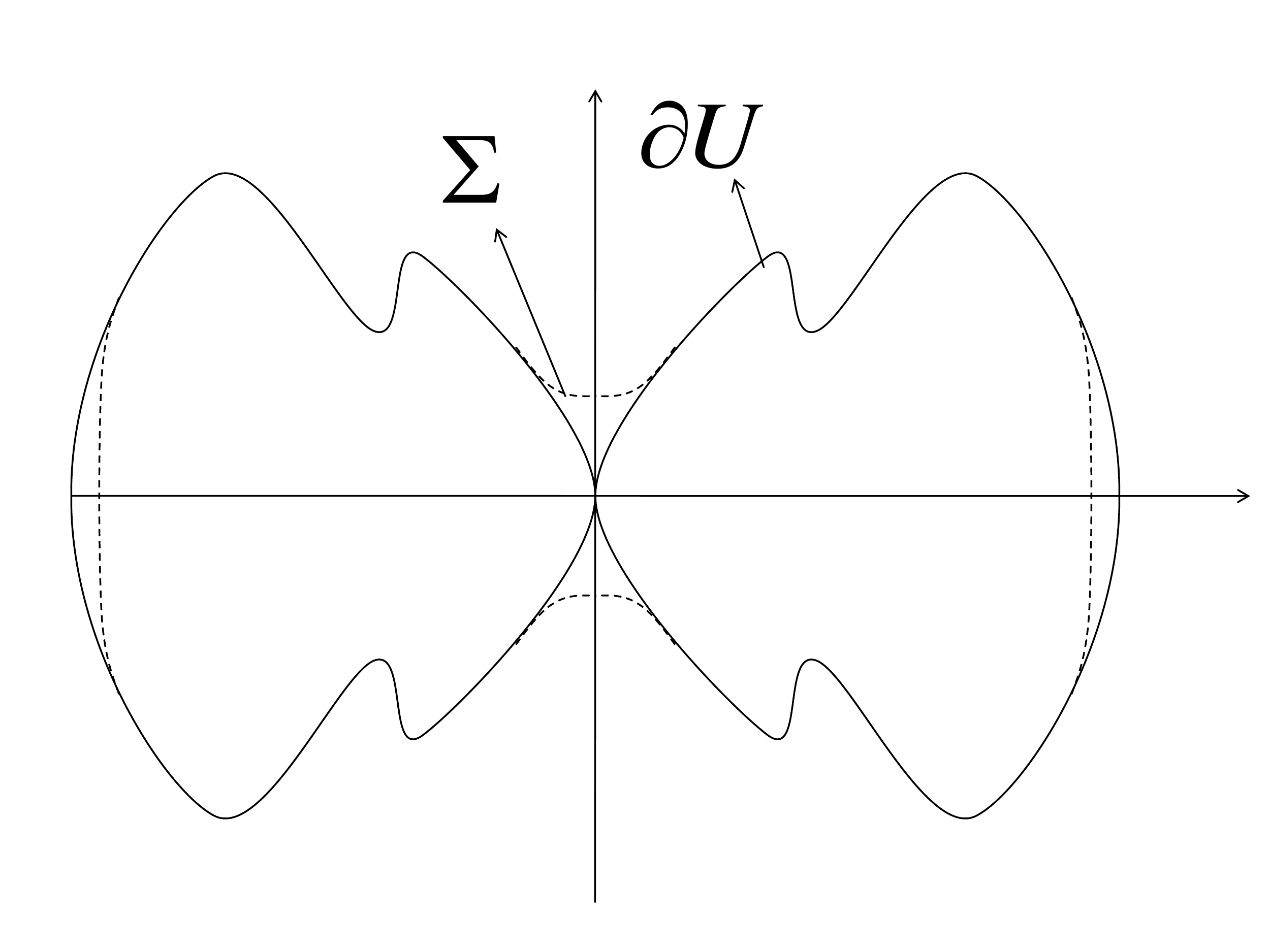}
		\vskip 0pt
		\caption{Proof of Proposition \ref{pro:sigma2}}
        \label{fig:sigma2}
	\end{center}
\end{figure}
\begin{proof} Because $U\cap\{x\geq0\}$ is an $\alpha$-domain, there exist $\delta_j$, $\gamma_j$ and $0<\delta_j<\gamma_j$ such that
$$
u_0(\delta_j)=u_0(\gamma_j)=u_0(-\delta_j)=u_0(-\gamma_j)=\frac{\alpha}{2^j}
$$
and
$$\partial U\cap C_{\alpha}=\{(x,y)\mid x=\pm v(y), |y|<\alpha\}\cup\{(x,y)\mid x=\pm w(y),|y|<\alpha\},$$
where $v,w\in C^{\infty}((-\alpha,\alpha))$ and $0<v(y)<w(y)$ for $|y|<\alpha$. Here $\delta_j$ is decreasing and $\gamma_j$ is increasing in $j$, respectively.

We let $w_j\in C^{\infty}((-\alpha/2^{j-1},\alpha/2^{j-1}))$ be a function satisfying
$$
w_j(y)=\left\{
\begin{array}{lcl}
\gamma_{j+2},\ 0\leq\dis{|y|<\frac{\alpha}{2^{j+1}}},\\
w(y),\ \dis{\frac{\alpha}{2^{j}}<|y|<\frac{\alpha}{2^{j-1}}},
\end{array}
\right.
$$
and let $u_j\in C^{\infty}((-\delta_{j-1},\delta_{j-1}))$ be a function satisfying
$$
u_j(x)=\left\{
\begin{array}{lcl}
\dis{\frac{\alpha}{2^{j+1}}},\ x\in[0,\delta_{j+2}],\\
u_0(x),\  x\in[\delta_j,\delta_{j-1}).
\end{array}
\right.
$$

Let $\Sigma_j$ consist of three parts: $\{(x,y)\mid |y|=u_j(x),\ x\in(-\delta_j,\delta_j)\}$, $\{(x,y)\mid x=\pm w_j(y), |y|<\alpha/2^j\}$ and $\partial U\cap \{|y|\geq \alpha/2^{j}\}$. It is easy to see that for sufficiently large $j$, $\Sigma_j\subset V\setminus \{(0,0)\}$ satisfies (i), (ii), (iii) for $c=\gamma_{j+2}$, $\rho=\alpha/2^{j+1}$ and $d=\delta_{j+2}$.
\end{proof}

Denote $\sigma(P,\alpha):\Sigma\times(-\delta,\delta)\rightarrow V$ ($V$ is given at the begining of this section and $\Sigma$ is given by Proposition \ref{pro:sigma2}) the flow generated by vector field $X$ in $\mathbb{R}^{n+1}$. Precisely, $\sigma(P,\alpha)$ is defined as following:
$$
\left\{
\begin{array}{lcl}
\dis{\frac{d\sigma(P,\alpha)}{d\alpha}=X(\sigma(P,\alpha))},\ P\in \Sigma,\\
\sigma(P,0)=P,\ \ \ \ \ P\in \Sigma.
\end{array}
\right.
$$

By (i) in Proposition \ref{pro:sigma2}, for any $C^{1}$ function $u:\Sigma\rightarrow\mathbb{R}$, ``the image of $u$ under $\sigma$''$:=\{\sigma(P,u(P))\mid P\in \Sigma\}$ is a $C^1$ hypersurface. Conversely, for any curve $\Gamma\subset V$ which is $C^1$ close to $\Sigma$, there exists a unique $C^1$ function $u:\Sigma\rightarrow\mathbb{R}$ such that $\Gamma=\{\sigma(P,u(P))\mid P\in \Sigma\}$. In other words, the map $\sigma(\cdot,t)$ defines new coordinates from $\Sigma$ to $V$. Therefore, if $\Gamma(t)\subset V$$(0<t<T)$ is a smooth family of smooth hypersurfaces and $C^1$ close to $\Sigma$, there exists a unique function $u \in C^\infty(\Sigma\times(0,T))$ such that $\Gamma(t)=\{\sigma(P,u(P,t))\mid P\in\Sigma\}$. Let $z=(z_1,\cdots,z_n)$ be the local coordinates on an open subset of $\Sigma$. If $\Gamma(t)$ evolves by $V=-\kappa+A$, under these coordinates $u$ satisfies the following equation
\begin{equation}\label{eq:para1}
\frac{\partial u}{\partial t}=a_{ij}(z,u,\nabla_z u)\nabla^2_{z_iz_j}u+b(z,u,\nabla_z u).
\end{equation}
Here $\{a_{ij}\}$ is a smooth, positive matrix. Precisely, we can see Section 3 in \cite{A2}. Consequently, (\ref{eq:para1}) is a parabolic equation. 

For example, $\sigma(\cdot,\alpha)$ is the flow defined as above.  We can easily deduce that 
$$
\sigma(P,\alpha)=\left\{
\begin{array}{lcl}
(x,\rho-\alpha),\ P\in B_d,\\
(-c+\alpha,y),\ P\in \Delta_{-c},\\
(c-\alpha,y),\ P\in \Delta_{c},
\end{array}
\right.
$$
where we choose the local coordinates:
\\
(1). $(x,\rho y)$ on $B_d$;
\\
(2). $(\pm c,y)$ on $\Delta_{\pm c}$. 

Suppose that $\Gamma(t)$ is symmetric to $x$-axis. Then, on $B_d$, $u$ depends only on $x$, $t$ and satisfies
\begin{equation}\label{eq:graph2}
u_t=\frac{u_{xx}}{1+u_x^2}+\frac{n-1}{\rho-u}-A\sqrt{1+u_x^2}.
\end{equation}
In this case, $b=\frac{n-1}{\rho-u}-A\sqrt{1+u_x^2}$. For $n\geq 2$, it is easy to see $b$ is not smooth at $u=\rho$. This is the most significant difference between the condition $n=1$ and condition $n\geq 2$.

On $\Delta_{\pm c}$, since $u$ depends only on $y=(y_1,\cdots,y_n)$, $u$ satisfies 
\begin{equation}\label{eq:graphminus}
u_t=\dis{\left(\delta_{ij}-\frac{u_{y_i}u_{y_j}}{1+|\nabla u|^2}\right)u_{y_iy_j}- A\sqrt{1+|\nabla u|^2}}.
\end{equation}

\begin{lem}\label{lem:max} For smooth function $v(x,t)$ on $V\times(0,T)$, where $V$ is a compact set, we denote $m(t)$ by $$
m(t)=\max\{v(x,t)\mid x\in V\}.
$$
Then there exists $P_t\in V$ such that $v(P_t,t)=m(t)$ and $m^{\prime}(t)=v_t(P_t,t)$ for $t>0$.
\end{lem}
This is a well known result, for example, seeing \cite{M}.

\begin{prop}\label{pro:uniq2} For $n\geq2$, let $\Gamma_j(t)$, $t\in[0,T]$ be two families of hypersurfaces with $\sigma^{-1}(\Gamma_j(t))$ the graph of $u_j(\cdot,t)$ for certain $u_j\in C(\Sigma\times[0,T])$, $j=1,2$. Let $D_j(t)$ be bounded open domain with $\partial D_j(t)=\Gamma_j(t)$ and assume that $D_j(t)$ are $\alpha(t)$-domain, $j=1,2$. Moreover, assume that $u_j$ are smooth on $\Sigma\times(0,T]$ and smooth on $\Sigma\setminus(\Delta_{\pm c}\cup B_d)\times[0,T]$. And suppose that $\rho-u_j$ are bounded from below on $\Sigma\setminus(\Delta_{\pm c}\cup B_d)\times[0,T]$. If 

(1). $\Gamma_j(t)$ evolves by (\ref{eq:cur});

(2). $\Gamma_1(0)=\Gamma_2(0)$; 

(3). $\int_0^{T}\frac{1}{\alpha^2(t)}dt<\infty$, 
\\
then $\Gamma_1(t)=\Gamma_2(t)$ for $0<t\leq T$.
\end{prop}

\begin{proof} Consider function $v(P,t)=u_1(P,t)-u_2(P,t)$. From our assumptions, there holds $v\in C(\Sigma\times[0,T])$ and $v$ is smooth on $(\Sigma\setminus\Delta_{\pm c}\cup B_d)\times[0,T]$ and smooth on $\Sigma\times(0,T]$. Moreover $v(P,0)\equiv0$. We define $M(t)=\max\{v(P,t)\mid P\in \Sigma\}$. Choose $P_t$ as in Lemma \ref{lem:max} such that $M(t)=v(P_t,t)$ and $M^{\prime}(t)=v_t(P_t,t)$.

\textbf{Case 1.} $P_t\in B_d$, $u_j$ satisfy
\begin{equation}
u_t=\frac{u_{xx}}{1+u_x^2}+\frac{n-1}{\rho-u}-A\sqrt{1+u_x^2}.\tag{\ref{eq:graph2}}
\end{equation}
Obviously, $v$ satisfies the following equation
$$
v_t=a^1(x,t)v_{xx}+b^1(x,t)v_x+c^1(x,t)v,
$$
where $a^1(x,t)>0$ and $c^1(x,t)=\frac{n-1}{(\rho-u_1)(\rho-u_2)}$. Since $v$ attains its maximum at $P_t$, then
$v_x(P_t,t)=0$ and $v_{xx}(P_t,t)\leq0$. So we have $v_t(P_t,t)\leq c^1(x,t)v$. By assumption that $D_j$ are $\alpha(t)$-domain, then $\rho-u_j>\alpha(t)$, $j=1,2$. Therefore, 
$$
v_t(P_t,t)\leq\frac{n-1}{\alpha^2(t)}v(P_t,t).
$$
Consequently, $M^{\prime}(t)\leq\frac{n-1}{\alpha^2(t)}M(t)$.

\textbf{Case 2.} $P_t\in \Sigma\setminus(\Delta_{\pm c}\cup B_d)$. Then we can choose coordinates $z$ on some neighbourhood of $P_t$ on $\Sigma$ and $u_j$ satisfy (\ref{eq:para1}). We may write this equation as $u_t=F(z,t,u,\nabla u, \nabla^2 u)$. Then $v$ satisfies
$$
v_t=a^2_{ij}(z,t)v_{z_iz_j}+b^2_i(z,t)v_{z_i}+c^2(z,t)v
$$
where
$$
c^2(z,t)=\int_{0}^1F_u(z,t,u^{\theta},\nabla u^{\theta},\nabla ^2u^{\theta})d\theta,
$$
where $u^{\theta}=(1-\theta)u_0+\theta u_1$ and $\{a^2_{ij}\}$ is a positive definite.

Since $v$ is smooth on $\Sigma\setminus(\Delta_{\pm c}\cup B_d)\times[0,T]$ and $\rho-u_j$ are bounded from below on $\Sigma\setminus(\Delta_{\pm c}\cup B_d)$, $0<t<T$, then there exists a positive constant $C$ such that $|c^2(z,t)|\leq C$. The constant $C$ may depend on the choice of local coordinates $z$. By compactness of $\Sigma$, $\Sigma$ has a finite covering consisting of neighborhoods of local coordinates, and we can choose $C$ independent of the choice of local coordinates. Since $\nabla v(P_t,t)=0$, $\{v_{z_iz_j}(P_t,t)\}$ is negative semi-definite, 
$$
v_t(P_t,t)\leq c(P_t,t)v(P_t,t)\leq Cv(P_t,t).
$$
Consequently, $M^{\prime}(t)\leq CM(t)$.

\textbf{Case 3.} $P_t\in \Delta_{\pm c}$. We only consider $P_t\in \Delta_{-c}$. Then in the $z$-coordinates of $\Delta_{-c}$, $u_j$ satisfy the full graph equation
\begin{equation}
u_t=\dis{\left(\delta_{ij}-\frac{u_{y_i}u_{y_j}}{1+|\nabla u|^2}\right)u_{y_iy_j}- A\sqrt{1+|\nabla u|^2}}.\tag{\ref{eq:graphminus}}
\end{equation}
 Hence $v=u_1-u_2$ satisfies a linear parabolic equation 
$$
v_t=a^3_{ij}(z,t)v_{z_iz_j}+b^3_i(z,t)v_{z_i},
$$
where $\{a^3_{ij}\}$ is positive definite. Obviously, $ \nabla v(P_t,t)=0$ and $\{v_{z_iz_j}(P_t,t)\}$ is negative semi-definite. It follows that $M^{\prime}(t)\leq0$.

From the three cases above, if we put
$$
r(t)=\frac{n-1}{\alpha^2(t)}+C,
$$
then there holds $M^{\prime}(t)\leq r(t)M(t)$. Consequently, by the assumption of $\alpha(t)$, $$\dis{M(t)\leq M(0)e^{\int_{0}^tr(t)ds}=0}.$$

By considering $m(t)=\min\{v(P,t)\mid P\in\Sigma\}$, we can similarly prove $m(t)\geq 0$. Therefore $\Gamma_1(t)=\Gamma_2(t)$ for $0\leq t\leq T$.
\end{proof}

Note that the intial curve in our problem is singular at $x$-axis. The assumption that ``$D_j(t)$ are $\alpha(t)$ domain'' in Proposition \ref{pro:uniq2} means that $\Gamma_j(t)$ ``escape'' from origin with speed $\alpha(t)$. If the ``escape speed'' satisfies 
$$
\int_0^{T}\frac{1}{\alpha^2(t)}dt<\infty,
$$
we can get the uniqueness.

Following Lemmas \ref{lem:closeas} and \ref{lem:closebou} can be proved by similar argument to the argument in $\mathbb{R}^2$ given by \cite{Z}. For reader's convenience, we give the proof.

\begin{lem}\label{lem:closeas} There exists a sequence of closed sets $E_j$ such that $E_j^{\circ}$ are $\alpha/2^j$-domains and $E_j\downarrow \overline{U}$. Here $U$ is given at the beginning of the section and $E^{\circ}$ denotes the interior of the set $E$.
\end{lem}
\begin{proof} We choose $\delta_j$ as in the proof of Proposition \ref{pro:sigma2}. We can construct even functions $v_j\in C^{\infty}((-b_0,b_0))$ such that 
$$
v_j(x)=\left\{
\begin{array}{lcl}
\alpha/2^j,\ x\in (-\delta_j/2,\delta_j/2),\\
u_0(x), x\in [-b_0,-\delta_j]\cup[\delta_j,b_0],
\end{array}
\right.
$$
$v_j(x)\geq u_0(x)$, $x\in[-b_0,b_0]$ and $v^{\prime}_j(x)>0$, $x\in(\delta_j/2,\delta_j)$. It is easy to see $v_j\downarrow u_0$ uniformly in $[-b_0,b_0]$ as $j\rightarrow \infty$.

Let $E_j=\{(x,y)\mid|y|\leq v_j(x),\ -b_0\leq x\leq b_0\}$. Since $v_j\downarrow u_0$ uniformly in $[-b_0,b_0]$, $E_j\downarrow \overline{U}$ as $j\rightarrow \infty$. It is easy to check $E_j^{\circ}$ are $\alpha/2^j$-domain.
\end{proof}

\begin{lem}\label{lem:closebou}  Let the same assumption of (1) in Theorem \ref{thm:equal90} be given. Then there exists $t_1>0$ such that, for all $t_2$ satisfying $0<t_2<t_1$, the second fundamental forms and derivatives of $\partial E_j(t)$ are uniformly bounded for $t_2\leq t\leq t_1$,  where $E_j(t)$ denote the closed evolution of $V=-\kappa+A$ with $E_j(0)=E_j$ and $E_j$ are chosen as in Lemma \ref{lem:closeas}.
\end{lem}
\begin{proof} 
Let $E_j(t)=\{(x,y)\mid|y|\leq v_j(x,t),\ -c_j(t)\leq x\leq c_j(t)\}$.

{\bf Step 1.} For all $t_2$ satisfying $0<t_2<\delta$ ($\delta$ given by assumption $(A-)$), there exists a constant $c>0$ such that 
$$
v_j(0,t)>c,\ t_2/2<t<\delta.
$$

Let $U^{+}(t)$ denote the bounded set with $\partial U^+(t)=\Lambda^+(t)$. Since $U^+(0)=U\cap\{x\geq0\}\subset E_j=E_j(0)$, $U^{+}(t)\subset E_j(t)$. By our assumption that $a_*(t)<0$ for $0<t\leq\delta$, $O\in U^{+}(t)\subset E_j(t)$ for $0<t<\delta$. For all $t_2\in(0,\delta)$, there exists $c>0$ such that $v_j(0,t)> c$ for $t_2/2\leq t\leq \delta$.

{\bf Step 2.} Construction of four auxiliary balls.

Since $U\cap\{x\geq0\}$ is an $\alpha$-domain, there exist $\beta_2>\beta_1>0$ such that $u_0(\pm\beta_1)=u_0(\pm\beta_2)=\alpha$ and $u_0^{\prime}(x)<0$ for $x>\beta_2$, $u_0^{\prime}(x)>0$ for $0<x<\beta_1$. There exist $p>\beta_1$ and $0<q<\beta_2$ such that $\dis{u_0(\pm q)=u_0(\pm p)=\frac{\alpha}{2}}$. We consider the points
$$Q=(-p,0),\ \ \ \ \ \ \ \ \ \ \ \ \ \ \ \ \ \ P=(p,0),$$
$$Q^{\prime}=(-p,\alpha),\ \ \ \ \ \ \ \ \ \ \ \ \ \ \ \ \ P^{\prime}=(p,\alpha).$$

\begin{figure}[htbp]
	\begin{center}
            \includegraphics[height=8cm]{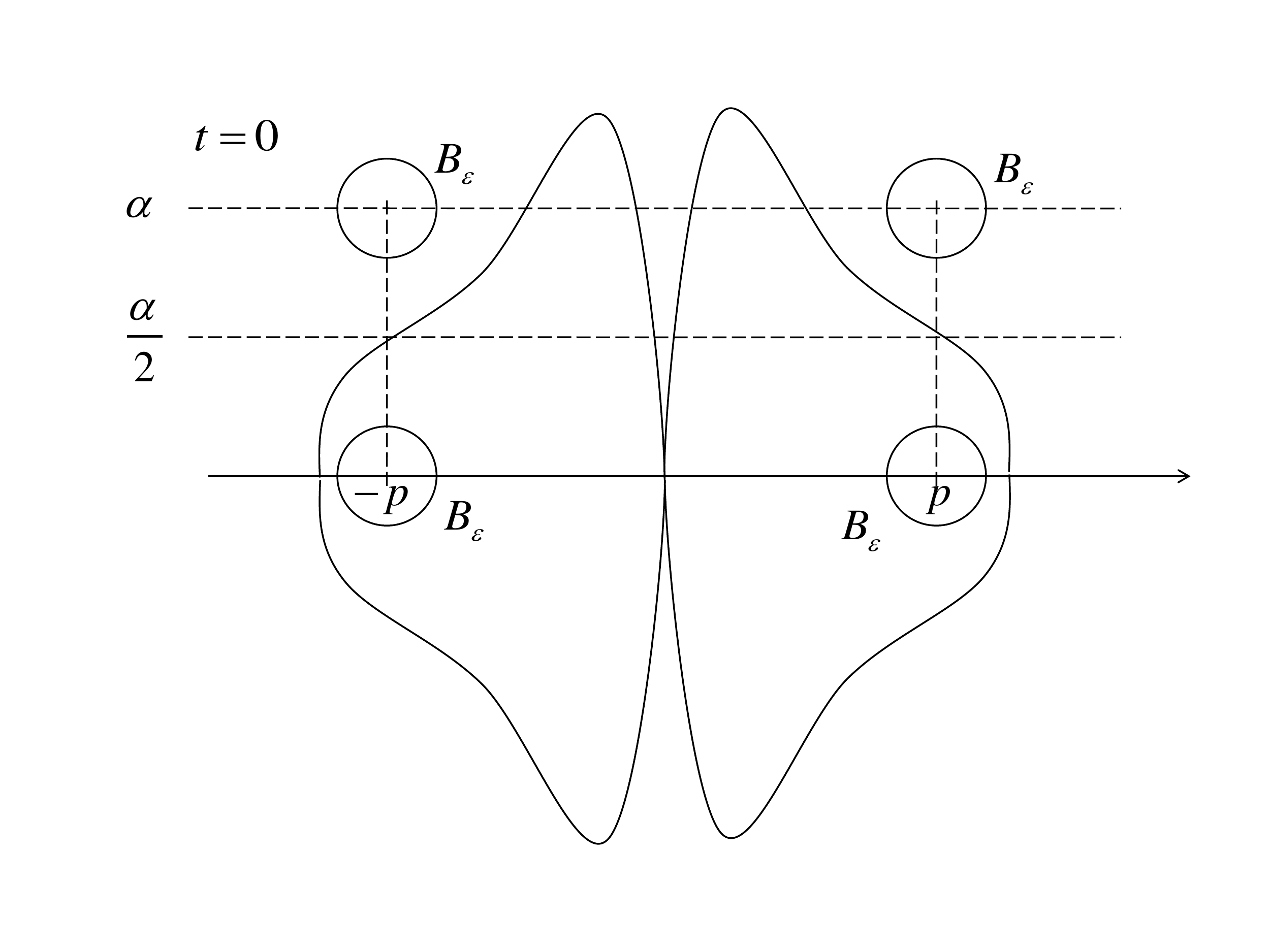}
		\vskip 0pt
		\caption{Proof of Lemma \ref{lem:closebou}}
        \label{fig:uniformb}
	\end{center}
\end{figure}

Since $P\in U$ and $P^{\prime}\in \overline{U}^{c}$, there exists $\epsilon$ such that $\overline{B_{\epsilon}(P)}\subset U$ and $\overline{B_{\epsilon}(P^{\prime})}\subset\overline{U}^c$. Consequently, $\overline{B_{\epsilon}(P)}\cup \overline{B_{\epsilon}(Q)}\subset E^{\circ}$ and $\overline{B_{\epsilon}(P^{\prime})}\cup \overline{B_{\epsilon}(Q^{\prime})}\subset E^{c}$. Then for $j$ large enough, $\overline{B_{\epsilon}(P)}\cup \overline{B_{\epsilon}(Q)}\subset E_j^{\circ}$ and $\overline{B_{\epsilon}(P^{\prime})}\cup \overline{B_{\epsilon}(Q^{\prime})}\subset E_j^{c}$. Comparison principle shows that
\begin{equation}\label{eq:q1}
\overline{B_{\epsilon(t)}(P)}\cup \overline{B_{\epsilon(t)}(Q)}\subset E_j(t)^{\circ}
\end{equation}
for $0<t<\delta_2$. By Theorem \ref{thm:sep},
\begin{equation}\label{eq:q2}
\overline{B_{\epsilon(t)}(P^{\prime})}\cup \overline{B_{\epsilon(t)}(Q^{\prime})}\subset E_j(t)^c
\end{equation}
for $0< t<\delta_2$. Here $\epsilon(t)$ is the solution of (\ref{eq:invercircle}) with $\epsilon(0)=\epsilon$ on the interval $[0,\delta_1)$. Take $\delta_2$ independent of $j$ such that $\epsilon(t)>\epsilon/2$ for $0<t<\delta_2$. 
\begin{figure}[htbp]
	\begin{center}
            \includegraphics[height=7cm]{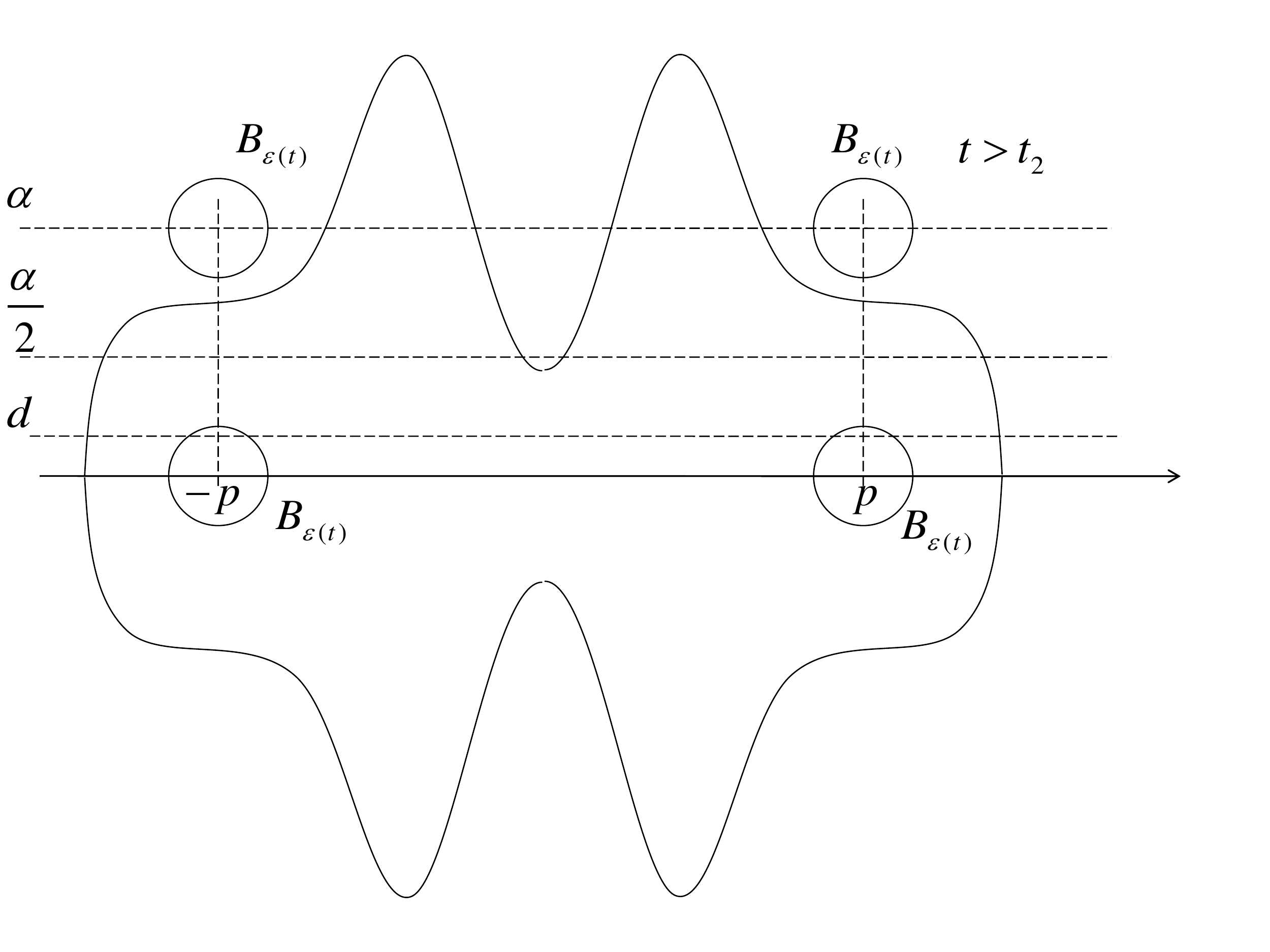}
		\vskip 0pt
		\caption{Proof of Lemma \ref{lem:closebou}}
        \label{fig:uniformb2}
	\end{center}
\end{figure}

{\bf Step 3.} Divide $\partial E_j(t)$ into two parts by auxiliary balls. 

Since for all $\rho<\alpha/2$, $C_{\rho}$ intersects $\partial E_j$ at most four times, by the intersection number argument as in the proof of Lemma \ref{lem:alphadn}, there exists $t_0>0$ such that $C_{\rho}$ intersects $\partial E_j(t)$ at most four times for $0<t<t_0$. By continuity, we can deduce that there exists $\delta_4$ such that for all $\rho<\alpha$, the equation $v_j(x,t)=\rho$ has just one root for $x>p$ for all $t<\delta_4$. By symmetry, it also holds for $x<-p$. 

Put $t_1=\min\{t_0,\delta_2,\delta_3,\delta_4\}$. Then Step 1 and intersection argument show that $E_j(t)^{\circ}$ are all $c$-domains for $t_2/2<t<t_1$.  Let $d<\min\{c,\epsilon/4\}$. By (\ref{eq:q1}) in Step 2, we have $v_j(x,t)>d$ for $t_2/2<t<t_1$ and $x$ with $|x-p|<\sqrt{\epsilon^2(t)-d^2}$ or $|x+p|<\sqrt{\epsilon^2(t)-d^2}$. By $\epsilon(t)>\epsilon/2$,
$$
v_j(x,t)\geq d\ \text{in}\ \Omega=(-p-\frac{\sqrt{3}}{4}\epsilon,p+\frac{\sqrt{3}}{4}\epsilon)\times(t_2/2,t_1).  
$$

For $x\leq-p$, by (\ref{eq:q2}) in Step 2,
$$
v_j(x,t)<\alpha/2-\epsilon(t)<\alpha/2-\epsilon/2\ \text{for}\ x\leq-p,\ 0\leq t<t_1.
$$
This is also true for $x\geq p$.

{\bf Step 4.} The derivatives and second fundamental forms of $\partial E_j(t)$ are bounded in $\Omega^{\prime}=[-p,p]\times(t_2,t_1)$.

Since $v_j(x,t)\geq d$ in $\Omega=(-p-\frac{\sqrt{3}}{4}\epsilon,p+\frac{\sqrt{3}}{4}\epsilon)\times(t_2/2,t_1)$, Theorem \ref{thm:grad} implies that $v_{jx}$ are uniformly bounded in $\Omega$. By Remark \ref{rem:hes}, $v_{jxx}$ are uniformly bounded in $\Omega^{\prime}$. 

{\bf Step 5.} The derivatives and second fundamental forms of $\partial E_j(t)$ are bounded for $x\leq -p$ and $x\geq p$, $t_2<t<t_1$.

We only consider for $x\leq-p$. For $0<t<t_1$, the part of $\partial E_j(t)$ on $x\leq-p$ can be represented by $x=w_j(y,t)$ for $|y|<\alpha/2$, $t\in (0,t_1)$, and $w_j$ satisfy the equation (\ref{eq:graph}) in the condition ``$-$'' and $n=1$. Then Corollary \ref{cor:es} and Remark \ref{rem:hes} imply that all $\frac{\partial ^k}{\partial y^k}w_j(y,t)$, $k=1,2$, are uniformly bounded for $|y|\leq\alpha/2-\epsilon/2$, $t_2<t<t_1$ and for any $t_2>0$. Then the derivatives and second fundamental forms of $\partial E_j(t)$ are uniformly bounded for $x\leq-p$, $t_2<t<t_1$. 

The proof of this lemma is completed.
\end{proof}

\begin{lem}\label{lem:opensy} There exists a sequence of open sets $\{U_j\}$ such that and $U_j\cap\{x\geq0\}$ being an $\alpha$-domain such that $U_j\uparrow U$ as $j\rightarrow\infty$.
\end{lem}

\begin{lem}\label{lem:openbou} Let the same assumption in Theorem \ref{thm:equal90} hold. Then there exists $t_1>0$ such that for all $t_2\in(0,t_1)$, the second fundamental forms and derivatives of $\partial U_j(t)$ is uniformly bounded for $t_2<t<t_1$, where $U_j(t)$ is the open evolution of $V=-\kappa+A$ with $U_j(0)=U_j$ and $\{U_j\}$ is chosen as in Lemma \ref{lem:opensy}.
\end{lem}

Lemmas \ref{lem:opensy} and \ref{lem:openbou} are all able to be proved as in \cite{Z}.

As mentioned in Proposition \ref{pro:uniq2}, in order to get the uniqueness, we must give the estimate of ``escape speed''. Let $R_0$ be taken small enough such that $B_{R_0}((R_0,0,\cdots,0))\cup B_{R_0}((-R_0,0,\cdots,0))\subset U$. In next lemma, we construct a sub-solution.

\begin{lem}\label{lem:subsolution}(Sub-solution)
Take $R_0$ as above. Function $\underline{u}$ is even and defined by
$$
\underline{u}(x,t)=\left\{
\begin{array}{lcl}
\sqrt{(r(t)+R(t))^2-R_0^2}-\sqrt{r^2(t)-x^2},\ 0\leq x<\frac{R_0r(t)}{R(t)+r(t)},\\
\sqrt{R^2(t)-(x-R_0)^2},\ \frac{R_0r(t)}{R(t)+r(t)}\leq x\leq R_0+R(t).
\end{array}
\right.
$$
Here $r(t)=t^{3/4}$ and $R(t)$ satisfies $R^{\prime}=A-n/R$, $R(0)=R_0$. Then there exists $t_*>0$ such that $(\underline{u},R_0+R(t))$ is a sub-solution of (\ref{eq:1graph}), (\ref{eq:1bounday}), (\ref{eq:1initial}), (\ref{eq:innerpositive}) for $0<t<t_*$.
\end{lem}

\begin{proof}
We can easily deduce that $|R(t)-R_0|=O(t)$, as $t\rightarrow0$. Since $r(t)=t^{3/4}>R(t)$, for sufficient small $t$, $\underline{u}$ is well-defined for small $t$.

1. Positive: Obviously, $u_0(x,t)>0$ for $-R_0-R(t)<x<R_0+R(t)$.

2. Initial condition: By the choice of $R_0$, $\underline{u}(x,0)=\sqrt{R_0-(x+R_0)^2}\leq u_0(x)$ for $0\leq x\leq R_0$ and $R_0\leq b_0$.

3. Boundary condition: Obviously, at boundary, 
$$
\underline{u}(-R_0-R(t),t)=\underline{u}(R_0+R(t),t)=0
$$
and
$$
\underline{u}_x(-R_0-R(t),t)=-\underline{u}_x(R_0+R(t),t)=\infty.
$$
4. Interior: For $\frac{R_0r(t)}{R(t)+r(t)}< x< R_0+R(t)$ or $-R_0-R(t)<x<-\frac{R_0r(t)}{R(t)+r(t)}$, $\underline{u}(x,t)=\sqrt{R^2(t)-(x-R_0)^2}$ satisfies (\ref{eq:1graph}). Next we only need prove $\underline{u}$ is a sub-solution of (\ref{eq:1graph}) for $-\frac{R_0r(t)}{R(t)+r(t)}<x<\frac{R_0r(t)}{R(t)+r(t)}$ and $t$ small. By calculation,
$$
1+\underline{u}_x^2=\frac{r^2}{r^2-x^2};
$$
$$
\underline{u}_{xx}=\frac{r^2}{(r^2-x^2)^{3/2}};
$$
$$
\underline{u}_t=\frac{(r+R)(r^{\prime}+R^{\prime})}{\sqrt{(r+R)^2-R_0^2}}-\frac{rr^{\prime}}{\sqrt{r^2-x^2}}.
$$
Therefore, 
\begin{eqnarray*}
\underline{u}_t&-&\frac{\underline{u}_{xx}}{1+\underline{u}_x^2}+\frac{n-1}{\underline{u}}-A\sqrt{1+\underline{u}_x^2}=\frac{(r+R)(r^{\prime}+R^{\prime})}{\sqrt{(r+R)^2-R_0^2}}-\frac{rr^{\prime}}{\sqrt{r^2-x^2}}-\frac{1}{\sqrt{r^2-x^2}}\\
&+&\frac{n-1}{\sqrt{(r+R)^2-R_0^2}-\sqrt{r^2-x^2}}-\frac{Ar}{\sqrt{r^2-x^2}}\leq\frac{(r+R)(r^{\prime}+R^{\prime})}{\sqrt{(r+R)^2-R_0^2}}-r^{\prime}-\frac{1}{r}\\
&+&\frac{n-1}{\sqrt{(r+R)^2-R_0^2}-r}.
\end{eqnarray*}
Since $|R^{\prime}(t)|$ is bounded, $R(t)$ is bounded from above and below for small $t$, and $r(t)=t^{3/4}$, we can deduce that 
$$
\frac{1}{\sqrt{(r+R)^2-R_0^2}}=\frac{1}{\sqrt{(r+R-R_0)(r+R+R_0)}}=O(t^{-3/8});
$$
$$
(r+R)(r^{\prime}+R^{\prime})=O(t^{-1/4})
$$
as $t\rightarrow0$. Consequently, 
\begin{eqnarray*}
\underline{u}_t&-&\frac{\underline{u}_{xx}}{1+\underline{u}_x^2}+\frac{n-1}{\underline{u}}-A\sqrt{1+\underline{u}_x^2}\leq C_1 t^{-5/8}-C_2 t^{-1/4}-C_3 t^{-3/4}+C_4 t^{-3/8}\\
&=&t^{-3/4}(C_1t^{1/8}-C_2t^{1/2}-C_3+C_4t^{3/8})<0
\end{eqnarray*}
for any sufficient small $t>0$. It is easy to check that $\underline{u}$ is $C^1$ at $x=\frac{R_0r(t)}{R(t)+r(t)}$. Then there exists $t_*>0$ such that $\underline{u}$ is a sub-solution of (\ref{eq:1graph}) in viscosity sense for $0<t<t_*$. 

We complete the proof.
\end{proof}

\begin{cor}\label{cor:escapespeed}
Recall $U=\{(x,y)\in\mathbb{R}^{n+1}\mid |y|<u_0(x),\ b_0<x<b_0\}$ and $E(t)$ be the outer evolution of $\overline{U}$. Then there exists $t^*>0$ such that, for each $t\in[0,t^*)$, $E(t)$ can be described as follows,
$$
E(t)=\{(x,y)\in\mathbb{R}^{n+1}\mid |y|\leq v(x,t),\ -b_1(t)\leq x\leq b_1(t)\}.
$$
Here $(v,b_1)$ is the uniqueness solution of (\ref{eq:1graph}), (\ref{eq:1bounday}), (\ref{eq:1initial}), (\ref{eq:innerpositive}) on the interval $[0,t^*)$. Moreover, there exists $\alpha(t)$ with $\int_0^{t^*}\frac{1}{\alpha^2(t)}dt<\infty$ such that $E(t)^{\circ}$ is $\alpha(t)$-domain for $t<t^*$.
\end{cor}
\begin{proof}
Let $E_j(t)$ be given by Lemma \ref{lem:closebou}. Since $\{(x,y)\in \mathbb{R}^{n+1}\mid|y|\leq\underline{u}(x,0)\}\subset E_j(0)=\{(x,y)\in \mathbb{R}^{n+1}\mid|y|\leq v_j(x,0),\ -c_j(0)\leq x\leq c_j(0)\})$, $\{(x,y)\in \mathbb{R}\times\mathbb{R}^N\mid|y|\leq\underline{u}(x,t)\}\subset E_j(t)=\{(x,y)\in \mathbb{R}^{n+1}\mid|y|\leq v_j(x,t),\ -c_j(t)\leq x\leq c_j(t)\}$ for $0<t<t_*$ and for all $j$. 

By $r(t)=t^{3/4}$, $|R(t)-R_0|=O(t)$ and boundedness of $R(t)$ from below as $t\rightarrow0$, there exists $t^*>0$ such that
\begin{eqnarray*}
\underline{u}(0,t)&=&\sqrt{(r(t)+R(t))^2-R_0^2}-\sqrt{r^2(t)}= \sqrt{r(t)}\left(\sqrt{r(t)+2R(t)+\frac{R(t)-R_0}{r(t)}}-\sqrt{r(t)}\right)\\
&\geq& C^{\prime}t^{3/8}-t^{3/4}\geq Ct^{3/8},
\end{eqnarray*}
for $t<t^*$. If we taken
$$
\alpha(t)=Ct^{3/8}\ \text{for}\ 0\leq t<t^*,
$$
then $\int_0^{t^*}\frac{1}{\alpha^2(t)}dt<\infty$.
Therefore, $E_j(t)^{\circ}$ are all $\alpha(t)$-domains. Moreover, $(v_j,c_j)$ satisfy
(\ref{eq:1graph}), (\ref{eq:1bounday}), (\ref{eq:innerpositive}). By $E_j(0)\downarrow E(0)$, Theorem \ref{thm:mon} and the same method as in Lemma \ref{lem:closebou}, we can show that $E_j(t)\downarrow E(t)$ and derivatives and second fundamental forms of $E_j(t)$ are uniformly bounded. $(v,b_1)=\lim\limits_{j\rightarrow\infty}(v_j,c_j)$ is the solution of (\ref{eq:1graph}), (\ref{eq:1bounday}), (\ref{eq:1initial}), (\ref{eq:innerpositive}). Moreover, $\{(x,y)\in \mathbb{R}^{n+1}\mid|y|\leq\underline{u}(x,t)\}\subset E(t)$ for $0<t<t^*$. $E(t)^{\circ}$ is also an $\alpha(t)$-domain $0<t<t^*$. The uniqueness of the solution follows from Proposition \ref{pro:uniq2}.
\end{proof}

\begin{proof}[Proof of (1) in Theorem \ref{thm:equal90}] Let $E(t)$ and $U(t)$ be the closed and open evolution of (\ref{eq:cur}) with $E(0)=\overline{U}$ and $U(0)=U$, respectively.

Let $U^{+}(t)$ denote the bounded open set with $\partial U^+(t)=\Lambda^+(t)$. Since $U^+(0)=U\cap\{x\geq0\}\subset U=U(0)$, there holds $U^{+}(t)\subset U(t)$. By our assumption that $a_*(t)<0$ for $0<t\leq\delta$, $O\in U^{+}(t)\subset U(t)$ for $0<t<\delta$, where $\delta$ is given in assumption $(A-)$. This means that $\partial U(t)$ escapes away from origin. Therefore, 
$$
\{(x,y)\in \mathbb{R}^{n+1}\mid|y|\leq\underline{u}(x,t)\}\subset U(t),\ 0<t<T,
$$
where $T=\min\{t^*,\delta, t_1\}$. (Recall $t_1$ is given by Lemma \ref{lem:openbou}) Consequently, $U(t)$ is also an $\alpha(t)$-domain. For small $t$, we can easily check that $\partial E(t)$ and $\partial U(t)$ satisfy the assumption in Proposition \ref{pro:uniq2}. We get $\partial E(t)=\partial U(t)$ for $0\leq t<T$.

We complete the proof.
\end{proof}

\begin{proof}[Proof of (2) in Theorem \ref{thm:equal90}] Choose $T_1<\min\{t^*,\delta\}$. Here $\delta$ is given in assumption $(A+)$, and $t^*$ is given in Corollary \ref{cor:escapespeed}.

Let $U^{\pm}(t)$ be the bounded open domain with $\partial U^{\pm}(t)=\Lambda^{\pm}(t)$. Thus the left end point of $U^+(t)$ and the right end point of $U^-(t)$ are $(a_*(t),0,\cdots,0)$ and $(-a_*(t),0,\cdot,0)$, respectively. By assumption $(A+)$, $-a_*(t)\leq a_*(t)$, $0\leq t<T_1$. Therefore, $U^+(t)\cap U^-(t)=\emptyset$, $0\leq t<\delta$. From Lemma \ref{lem:sep}, the inner evolution $U(t)$ satisfies $U(t)=U^{+}(t)\cup U^{-}(t)$, for $0\leq t<\delta$.

Corollary \ref{cor:escapespeed} shows that $E(t)^{\circ}$ is an $\alpha(t)$-domain for $0<t<T_1$. By $\alpha(t)= Ct^{3/8}$, there exists small enough $\rho$ such that the ball 
$$
B_{\rho}((0,C(\frac{T_1}{4})^{3/8},0,\cdots,0))\subset E(t), \frac{T_1}{2}<t<T_1
$$
and 
$$
\dis{B_{\rho}((0,C(\frac{T_1}{4})^{3/8},0,\cdots,0))}\cap U(t)=\emptyset,\ \dis{\frac{T_1}{2}<t<T_1}.
$$
This means that the interface evolution $\Gamma(t)=E(t)\setminus U(t)$ has interior.
\end{proof}

\begin{rem}\label{rem:conse1}
(1) In the proof of (1) in Theorem \ref{thm:equal90}, we get the closed evolution and the open evolution are all connected sets. Therefore, they are homeomorphic. Moreover, using the unique result, we can prove that they are coincide.

(2) In the proof of (2) in Theorem \ref{thm:equal90}, we get the closed evolution is connected and the open evolution is separated. Therefore, they are not homeomorphic.
\end{rem}

\section{Proof of Theorem \ref{thm:less90}}
In this section, we give the proof of Theorem \ref{thm:less90}. First, we introduce the following similarity transformation: for $T>0$,
\begin{equation}\label{eq:transformation1}
z=\frac{x}{\sqrt{2(T-t)}},\ \tau=-\frac{1}{2}\ln(T-t)
\end{equation}
and
\begin{equation}\label{eq:transformation2}
w(z,\tau)=\frac{1}{\sqrt{2}}e^{\tau}u(\sqrt{2}e^{-\tau}z,T-e^{-2\tau}).
\end{equation}

Then $u$ satisfies 
\begin{equation}
u_t=\frac{u_{xx}}{1+u_x^2}-\frac{n-1}{u}+A\sqrt{1+u_x^2},\tag{\ref{eq:1graph}}
\end{equation}
 if and only if $w$ satisfies
\begin{equation}\label{eq:shrinkingeq1}
w_{\tau}=\frac{w_{zz}}{1+w_z^2}-zw_z+w-\frac{n-1}{w}+\sqrt{2}Ae^{-\tau}\sqrt{1+w_z^2}.
\end{equation}
In 1992, \cite{A3} shows that there is a torus shape self-similar solution of (\ref{eq:cur}) for $A=0$ called ``Angenent shrinking doughnut''. The self-similar solution remaining the shape of doughnut shrinks to a point. Moreover, in \cite{AAG}, using this self-similar solution, they prove that after a rotational hypersurface pinches, the hypersurface will be separated into two disjoint components. We also expect to prove Theorem 1.2 by using some self-similar solution of (\ref{eq:cur}), however, it is difficult to find such solution. Therefore, we construct a compact self-similar super-solution of (\ref{eq:1graph}).

\begin{prop}\label{prop:supersolution}(Super-solution of equation (\ref{eq:shrinkingeq1}))
Denote $\overline{w}:=C-\sqrt{\rho^2-z^2}$, $-\rho\leq z\leq \rho$. 

For $n>2$, for every $C$, $\rho$ with $C^2+\rho^2<n$ and $C>\rho>1$, there exists $\tau_0(C,\rho)$ such that $\overline{w}$ is a super-solution of (\ref{eq:shrinkingeq1}) for $-\rho< z<\rho$, $\tau>\tau_0$.

For $n=2$, Fix $\theta\in(0,1)$, $\epsilon_0\in(0,\frac{2}{9}\theta)$ arbitrary. Then, for each $1+\theta\epsilon_0<\rho<C<1+\epsilon_0$, there exists $\tau_0(C,\rho)$ such that $\overline{w}$ is a super-solution of (\ref{eq:shrinkingeq1}) for $-\rho< z<\rho$, $\tau>\tau_0$.
\end{prop}
\begin{proof}
By calculation,
$$
\overline{w}_{\tau}=0,
$$
$$
\overline{w}_z=\frac{z}{\sqrt{\rho^2-z^2}}
$$
and
$$
\overline{w}_{zz}=\frac{\rho^2}{(\rho^2-z^2)^{3/2}}.
$$
For convenience, we put $q=\sqrt{\rho^2-z^2}$ for $-\rho\leq z\leq \rho$.
\begin{eqnarray*}
\overline{w}_{\tau}&-&\frac{\overline{w}_{zz}}{1+\overline{w}_z^2}+z\overline{w}_z-\overline{w}+\frac{n-1}{\overline{w}}-\sqrt{2}Ae^{-\tau}\sqrt{1+\overline{w}_z^2}\\
&=&-\frac{1}{q}+\frac{n-1}{C-q}-C+q+\frac{\rho^2-q^2}{q}-\sqrt{2}Ae^{-\tau}\frac{\rho}{q}\\
&=&-\frac{1}{q}+\frac{n-1}{C-q}-C+\frac{\rho^2}{q}-\sqrt{2}Ae^{-\tau}\frac{\rho}{q}\\
&=&\frac{1}{q(C-q)}\left(C q^2-(\rho^2+C^2-n)q-C+\rho^2C-\sqrt{2}Ae^{-\tau}\rho(C-q)\right).
\end{eqnarray*}

For $n>2$, if $\rho^2+C^2<n$ and $\rho<C$, we can deduce $C q^2-(\rho^2+C^2-n)q-C+\rho^2C$ of the right hand side of the formula above attains its minimum at $q=0$.
Consequently,
$$
\overline{w}_{\tau}-\frac{\overline{w}_{zz}}{1+\overline{w}_z^2}+z\overline{w}_z-\overline{w}+\frac{n-1}{\overline{w}}-\sqrt{2}Ae^{-\tau}\sqrt{1+\overline{w}_z^2}\geq \frac{1}{q(C-q)}(-C+\rho^2C-\sqrt{2}Ae^{-\tau}\rho C).
$$
Therefore, if $\rho>1$, then there exits $\tau_0(C,\rho)$ such that 
$$
\overline{w}_{\tau}-\frac{\overline{w}_{zz}}{1+\overline{w}_z^2}+z\overline{w}_z-\overline{w}+\frac{n-1}{\overline{w}}-\sqrt{2}Ae^{-\tau}\sqrt{1+\overline{w}_z^2}>0
$$
for $\tau>\tau_0$.

For $n=2$, $C q^2-(\rho^2+C^2-2)q-C+\rho^2C$ attains its minimum at $\frac{\rho^2+C^2-2}{2C}$. Consequently,
\begin{eqnarray*}
\overline{w}_{\tau}&-&\frac{\overline{w}_{zz}}{1+\overline{w}_z^2}+z\overline{w}_z-\overline{w}+\frac{1}{\overline{w}}-\sqrt{2}Ae^{-\tau}\sqrt{1+\overline{w}_z^2}\\
&\geq&
\frac{1}{q(C-q)}(-\frac{(C^2+\rho^2-2)^2}{4C}+(\rho^2-1)C-\sqrt{2}Ae^{-\tau}\rho C).
\end{eqnarray*}
Then fixed $\theta\in(0,1)$, for any $\epsilon_0\in(0,\frac{2}{9}\theta)$ and $1+\theta\epsilon_0<\rho<C<1+\epsilon_0$,  
$$
-\frac{(C^2+\rho^2-2)^2}{4C}+(\rho^2-1)C\geq \frac{8\theta\epsilon_0-36\epsilon_0^2}{4(1+\theta\epsilon_0)}=\frac{2\theta\epsilon_0-9\epsilon_0^2}{(1+\theta\epsilon_0)}>0.
$$
Therefore, there exists $\tau_0(C,\rho)>0$ such that 
$$
\overline{w}_{\tau}-\frac{\overline{w}_{zz}}{1+\overline{w}_z^2}+z\overline{w}_z-\overline{w}+\frac{1}{\overline{w}}-\sqrt{2}Ae^{-\tau}\sqrt{1+\overline{w}_z^2}>0
$$
for $\tau>\tau_0$.
\end{proof}

\begin{rem}
Under the condition $n>2$, in the proof of Proposition \ref{prop:supersolution}, for convenience, we assume 
$$
\rho^2+C^2<n.
$$
Indeed, it is not necessary. In the proof, we can use that $C q^2-(\rho^2+C^2-n)q-C+\rho^2C$ attains its minimum at $\frac{\rho^2+C^2-n}{2C}$.
\end{rem}

\begin{cor}\label{cor:selfsimilarsuper}
Let $\overline{w}$ and $\tau_0$ be given by Proposition \ref{prop:supersolution}. Then for $T<e^{-\tau_0}$, 
$$
\overline{u}(x,t;T)=\sqrt{2(T-t)}\overline{w}\left(\frac{x}{\sqrt{2(T-t)}}\right)
$$
is a super-solution of equation (\ref{eq:1graph}) for $-\rho\sqrt{2(T-t)}<x<\rho\sqrt{2(T-t)}$, $0<t<T$.
\end{cor}
This result is obvious by Proposition \ref{prop:supersolution}. Here we omit the proof.

\begin{rem}\label{rem:envelop}
For $n=2$, recall $\gamma$ given in Section 1 and $\epsilon_0$ given by Proposition \ref{prop:supersolution}. Then for all $0\leq \gamma<\arctan \sqrt{(1+\epsilon_0)^2-1}$, we can choose $\rho$ and $C$ satisfying $1<\rho<C<1+\epsilon_0$ such that
$$
\gamma<\frac{\sqrt{C^2-\rho^2}}{\rho}.
$$

For $n>2$, if $0\leq\gamma<\arctan\sqrt{n-2}$, we can choose $\rho$, $C$ satisfying $1<\rho<C$ and $\rho^2+C^2<n$ such that
$$
\gamma<\frac{\sqrt{C^2-\rho^2}}{\rho}.
$$
It is obvious that the cone 
$$
|y|=\frac{\sqrt{C^2-\rho^2}}{\rho}|x|
$$
is the envelop of the family of hypersurfaces $\{|y|=\lambda\overline{w}(x/\lambda)\}_{\lambda>0}$.

By the property of $u_0$, we can get
$$
y=\frac{\sqrt{C^2-\rho^2}}{\rho}|x|> u_0(x),
$$
for small $|x|$.

Therefore, there exists $T(\rho,C)$ such that for all $0\leq T<T(\rho,C)$, 
$$
\overline{u}(x,0;T)>u_0(x),
$$
$-\rho\sqrt{2T}<x<\rho\sqrt{2T}$.
\end{rem}
\begin{proof}[Proof of Theorem \ref{thm:less90}] 
As mentioned in Remark \ref{rem:envelop}, we choose 
$$
\alpha_n=\left\{
\begin{array}{lcl}
\arctan\sqrt{(1+\epsilon_0)^2-1},\ n=2,\\
\arctan\sqrt{n-2},\ n>2.
\end{array}
\right.
$$ 
Obviously, $\alpha_n\rightarrow \pi/2$, as $n\rightarrow\infty$.

For every $\gamma<\alpha_n$, we choose $C$ and $\rho$ as in Remark \ref{rem:envelop}. Let $T_{\gamma}<\min\{T(\rho,C),e^{-\tau_0(\rho,C)}\}$, where $\tau_0(\rho,C)$ is given by Proposition \ref{prop:supersolution} and $T(\rho,C)$ is given by Remark \ref{rem:envelop}. Next we show that for every $0<t<T_{\gamma}$, the origin $O\in E(t)^c$. Let flow 
$$
\overline{\Gamma}(t;T)=\{(x,y)\mid|y|=\overline{u}(x,t,T), -\rho\sqrt{2(T-t)}\leq x\leq \rho\sqrt{2(T-t)}\},
$$
$0\leq t<T$. Remark \ref{rem:envelop} shows that for every $0<t<T_{\gamma}$,
$$
\overline{\Gamma}(0;t)\cap E(0)=\emptyset.
$$ 
By comparison principle, we can easily show that 
$$
\overline{\Gamma}(s;t)\cap E(s)=\emptyset,\ 0\leq s\leq t.
$$
(Noting that $\overline{\Gamma}(s,t)$ is a hypersurface with boundary, we must consider the comparison principle in interior and boundary separately. By comparison principle, $\partial E(s)$ can not touch $\overline{\Gamma}(s,t)$ interior. At the boundary, gradient of $\overline{u}(x,s,t)$ is infinity. Theorem \ref{thm:grad} and \ref{thm:gu} implies that $\partial E(s)$ can not touch $\overline{\Gamma}(s,t)$ at the boundary. The details are left to the reader) Especially, 
$$
\{O\}\cap E(t)=\overline{\Gamma}(t;t)\cap E(t)=\emptyset.
$$
Therefore for every $0<t<T_{\gamma}$, there holds $O\in E(t)^c$. 

Here we show that $E(t)$ is separated into two connected components, for $0<t<T_{\gamma}$. Let $E^+(t)$ ($E^-(t)$) and $U^+(t)$ ($U^-(t)$) be the outer evolution and inner evolution of $V=-\kappa+A$ with $E^+(0)=\overline{U}\cap\{x\geq0\}$ ($E^-(0)=\overline{U}\cap\{x\leq 0\}$) and $U^+(0)=U\cap\{x\geq0\}$ ($U^-(0)=U\cap\{x\leq0\}$), respectively.

Since $U\cap\{x\geq0\}$ and $U\cap \{x\leq 0\}$ are $\alpha$-domains and $E^+(t)\cap E^-(t)=\emptyset$, using Theorem \ref{thm:partialUmeancurvature}, we obtain 
$$
 \partial E^+(t)=\partial U^+(t),\ \partial E^-(t)=\partial U^-(t),\ 0<t<T_{\gamma}.
$$ 
Therefore 
$$
\partial E(t)=\partial E^+(t)\cup \partial E^-(t)=\partial U^+(t)\cup \partial U^-(t)=\partial U(t),\ 0<t<T_{\gamma}.
$$
Here we complete the proof.
\end{proof}

\begin{cor}\label{cor:nonexist}
Let $u_0$ be a function as in Section 1 and let $\gamma$ be the constant in (\ref{eqn:contactangle}). For $n\geq 2$ and $0\leq\gamma<\alpha_n$, there is no solution of the following free boundary problem,
\begin{equation}
u_t=\frac{u_{xx}}{1+u_x^2}-\frac{n-1}{u}+A\sqrt{1+u_x^2},\ -b(t)<x<b(t),\ 0<t<T_1,\tag{\ref{eq:1graph}}
\end{equation}
\begin{equation}
u(b(t),t)=u(-b(t),t)=0,\ u_x(-b(t),t)=-u_x(b(t),t)=\infty,\ 0<t<T_1,\tag{\ref{eq:1bounday}}
\end{equation}
\begin{equation}
u(x,0)=u_0(x),\ -b_0\leq x\leq b_0,\tag{\ref{eq:1initial}}
\end{equation}
\begin{equation}
u(x,t)>0,\ -b(t)< x< b(t),\ 0<t<T_1.\tag{\ref{eq:innerpositive}}
\end{equation}
\end{cor}
\begin{proof}
Assume that there exists the solution $(u,b)$ of the free boundary problem. We can use the approximate argument similarly as in Lemma \ref{lem:closeas} to prove that the outer evolution $E(t)$ is written as follows
$$
 E(t)=\{(x,y)\mid\mathbb{R}^{n+1}\mid |y|\leq u(x,t), -b(t)\leq x\leq b(t)\}.
$$
This contradicts that $E(t)$ is separated into two connected components.
\end{proof}

\section{Proof of Theorem \ref{thm:less90in2dim}}
In this section, we give the proof of Theorem \ref{thm:less90in2dim}. 

\begin{prop}\label{prop:outerless}(Connected Outer evolution) Let $u_0$ be a function as in Section 1. In the plane, there is $T_1>0$ such that $(u,b)$ is a unique solution of the following free boundary problem,
\begin{equation}
u_t=\frac{u_{xx}}{1+u_x^2}+A\sqrt{1+u_x^2},\ -b(t)<x<b(t),\ 0<t<T_1,\tag{\ref{eq:1graph}*}
\end{equation}
\begin{equation}
u(b(t),t)=u(-b(t),t)=0,\ u_x(-b(t),t)=-u_x(b(t),t)=\infty,\ 0<t<T_1,\tag{\ref{eq:1bounday}}
\end{equation}
\begin{equation}
u(x,0)=u_0(x),\ -b_0\leq x\leq b_0,\tag{\ref{eq:1initial}}
\end{equation}
\begin{equation}
u(x,t)>0,\ -b(t)< x< b(t),\ 0<t<T_1.\tag{\ref{eq:innerpositive}}
\end{equation}

Moreover, the outer evolution
$$
E(t)=\{(x,y)\in\mathbb{R}^2\mid|y|\leq u(x,t), -b(t)\leq x\leq b(t)\},\ 0<t<T_1.
$$
\end{prop}
To prove this proposition, we need the following lemma.

\begin{lem}\label{lem:alphad2}(Lemma 4.7 in \cite{Z}) Let $U$ be an $\alpha$-domain in $\mathbb{R}^2$ and let $U(t)$ be the open evolution with $U(0)=U$. Then there exists $t_U>0$ such that $U(t)$ is an $(\alpha+At)$-domain, $0<t<t_U$.
\end{lem}

\begin{proof}
Using the approximate argument as in Lemma \ref{lem:closeas}, we can prove that there exists $T_1>0$ such that $E(t)^{\circ}$ is $At$-domain, $0<t<T_1$. Moreover, 
$$
E(t)=\{(x,y)\in\mathbb{R}^2\mid|y|\leq u(x,t), -b(t)\leq x\leq b(t)\},\ 0\leq t<T_1.
$$
Here $(u,b)$ is the unique solution of (\ref{eq:1graph}*), (\ref{eq:1bounday}), (\ref{eq:1initial}), (\ref{eq:innerpositive}). For the precise proof, we can see \cite{Z} similarly. Here we omit the details.
\end{proof}

\begin{rem}\label{rem:positiveouter}
Indeed, it is determined by the existence of the solution $(u,b)$ of (\ref{eq:1graph}*), (\ref{eq:1bounday}), (\ref{eq:1initial}), (\ref{eq:innerpositive}) whether the outer evolution is connected or separated. 

Under the condition $n=1$, saying roughly, if $u$ satisfies 
\begin{equation}
u_t=\frac{u_{xx}}{1+u_x^2}+A\sqrt{1+u_x^2}, \tag{\ref{eq:1graph}*}
\end{equation} 
and $u(x,0)=u_0(x)\geq0$, by comparison principle, $u(x,t)>0$, $t>0$. This means that the problem always has a ``positive'' solution in the plane. This can be explained precisely by Lemma \ref{lem:alphad2}

However, under the conditon $n\geq2$, the equation
\begin{equation}
u_t=\frac{u_{xx}}{1+u_x^2}-\frac{n-1}{u}+A\sqrt{1+u_x^2} \tag{\ref{eq:1graph}}
\end{equation}
has the ``contraction power $-\frac{n-1}{u}$''. We can not ensure that the problem has a ``positive'' solution with $u(x,0)=u_0(x)\geq0$. Lemma \ref{lem:subsolution} shows that if $\gamma=\pi/2$, this problem has a unique ``positive'' solution. 
\end{rem}

\begin{prop}\label{prop:innerless}(Separated inner evolution)
Suppose $0\leq \gamma<\pi/2$. Let $(u,a,b)$ be the solution of the following problem
\begin{equation}
u_t=\frac{u_{xx}}{1+u_x^2}+A\sqrt{1+u_x^2},\ a(t)<x<b(t),\ 0<t<T_1,\tag{\ref{eq:1graph}**}
\end{equation}
\begin{equation}
u(b(t),t)=u(a(t),t)=0,\ u_x(a(t),t)=-u_x(b(t),t)=\infty,\ 0<t<T_1,\tag{\ref{eq:1bounday}*}
\end{equation}
\begin{equation}
u(x,0)=u_0(x),\ 0\leq x\leq b_0,\tag{\ref{eq:1initial}*}
\end{equation}
\begin{equation}
u(x,t)>0,\ a(t)< x< b(t),\ 0<t<T_1.\tag{\ref{eq:innerpositive}*}
\end{equation}
Then there exists $\delta>0$ such that $a(t)>0$, $0<t<\delta$. Moreover, the inner evolution $U(t)$ can be written as follows,
$$
U(t)=U^+(t)\cup U^-(t),\ 0<t<\delta,
$$
where $U^+(t)=\{(x,y)\in\mathbb{R}^2\mid|y|<u(x,t),\ a(t)\leq x\leq b(t)\}$ and $U^-(t)=\{(-x,y)\mid (x,y)\in U^+(t)\}$.
\end{prop}

\begin{proof}
We claim that there exists $\delta>0$ such that $a(t)>0$ for $0<t<\delta$. If the claim holds, $U^+(t)\cap U^-(t)=\emptyset$, $0<t<\delta$. Using Lemma \ref{lem:sep}, $U(t)=U^+(t)\cap U^-(t)$, $0<t<\delta$.

We give the sketch of the proof of the claim.  

Let $\gamma<\gamma_1<\pi/2$. Define a family of circles 
$$
v_{\lambda}(y)=\lambda C-\sqrt{(\lambda C\cos(\pi/2-\gamma_1))^2-y^2}.
$$
It is easy to find that the envelop of $\{v_{\lambda}\}_{\lambda>0}$ is $|y|=\tan\gamma_1 x$.

Let $\{(x,y)\mid x=v_0(y), -\delta_0<y<\delta_0\}$ be the left cap of $\partial U\cap\{x\geq0\}$.

By the choice of $\gamma_1$, if necessary, choose $\delta_0$ smaller such that 
$$
v_0(y)\leq \tan(\pi/2-\gamma_1)|y|,\ -\delta_0<y<\delta_0.
$$ 
Consider the following inverse equation
$$
v_t=\frac{v_{yy}}{1+v_y^2}-A\sqrt{1+v_y^2},\ -\delta_0<y<\delta_0,
$$
with $v(y,0)=\tan(\pi/2-\gamma_1)|y|$. Since the initial function is not smooth, we modify it by the family $\{v_{\lambda}\}_{\lambda>0}$ near the origin. Let $v_{\lambda}(y,t)$ be the solution with the modified initial data. We calculate 
$$
\frac{\partial}{\partial t}v_{\lambda}(0,0)=\frac{1}{\lambda C\cos(\pi/2-\gamma_1)}-A\rightarrow\infty,
$$
as $\lambda\rightarrow0$. Therefore, there exists constant $C>0$ such that $v_{\lambda}(0,t)>Ct$, for $t$ small. It is easy to see $v_{\lambda}(0,t)\rightarrow a(t)$, for every $t$ small. Then
$$
a(t)>Ct,\ 0<t<\delta,
$$
for some $\delta$. Here we complete the proof.

\end{proof}
\begin{proof}[Proof of Theorem \ref{thm:less90in2dim}]
This result is an easy consequence of Proposition \ref{prop:outerless} and Proposition \ref{prop:innerless}.
\end{proof}

\section{Appendix }
In this section, we prove there exists a unique smooth family of smooth hypersurfaces $\Gamma(t)$ satisfying
\begin{equation}\label{eq:hcur}
V=-\kappa+A,\ \text{on}\ \Gamma(t)\subset \mathbb{R}^{n+1},
\end{equation}
where $\Gamma(0)=\partial U$ and $U$ is an $\alpha$-domain.

Since $\partial U$ is not necessary smooth, we also use the level set method and prove the interface evolution is not fattening.

For $\alpha$-domain $U$, we choose a smooth vector field $X:\mathbb{R}^{n+1}\rightarrow\mathbb{R}^{n+1}$ such that
\\
(i) At any point $P\in\partial U$ which is not on the $x$-axis $\langle X(P), \textbf{n}(P)\rangle>0$, where $\textbf{n}$ is the inward unit normal vector at $P$.
\\
(ii) Near the two end points of $\partial U$, $X$ is constant vector with $X\equiv\pm e_0=(\pm1,0,\cdots,0).$

Since $X\neq0$ on the compact set $\partial U$, there are a constant $\delta>0$ and an open neighbourhood $V\supset\partial U$ on which $|X|\geq\delta>0$.

\begin{figure}[htbp]
	\begin{center}
            \includegraphics[height=4cm]{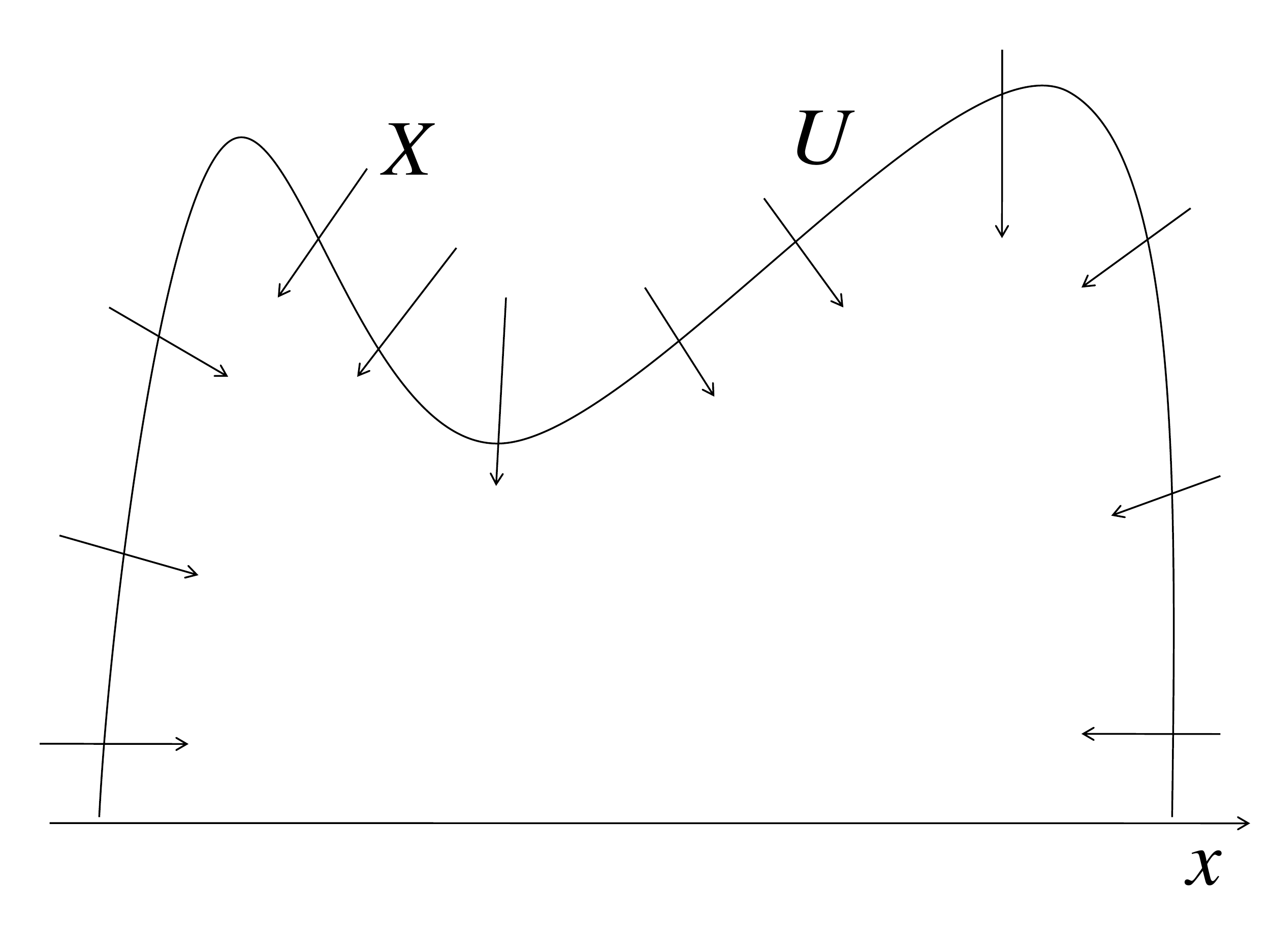}
		\vskip 0pt
		\caption{Vector field $X$}
        \label{fig:vectorX}
	\end{center}
\end{figure}

\begin{prop}\label{pro:sigma1}For sufficiently small $\rho>0$, there exists a smooth hypersurface $\Sigma\subset V$ with\\
(i) $X(P)\notin T_P\Sigma$ at all $P\in\Sigma$, i.e., $\Sigma$ is transverse to the vector field $X$.\\
(ii) $\Sigma=\partial U$ in $\{(x,y)\in\mathbb{R}\times\mathbb{R}^{n}\mid |y|\geq2\rho\}$.\\
(iii) $\Sigma\cap\{(x,y)\in\mathbb{R}\times\mathbb{R}^{n}\mid|y|\leq\rho\}$ consists of two flat disks $\Delta_a=\{(a,y)\in\mathbb{R}\times\mathbb{R}^{n}\mid|y|\leq\rho\}$ and $\Delta_b=\{(b,y)\in\mathbb{R}\times\mathbb{R}^{n}\mid|y|\leq\rho\}$ for some $a<b$.
\end{prop}

Seeing Figure \ref{fig:sigma1}, this proposition can be proved as in Proposition \ref{pro:sigma2}.

\begin{figure}[htbp]
	\begin{center}
            \includegraphics[height=4cm]{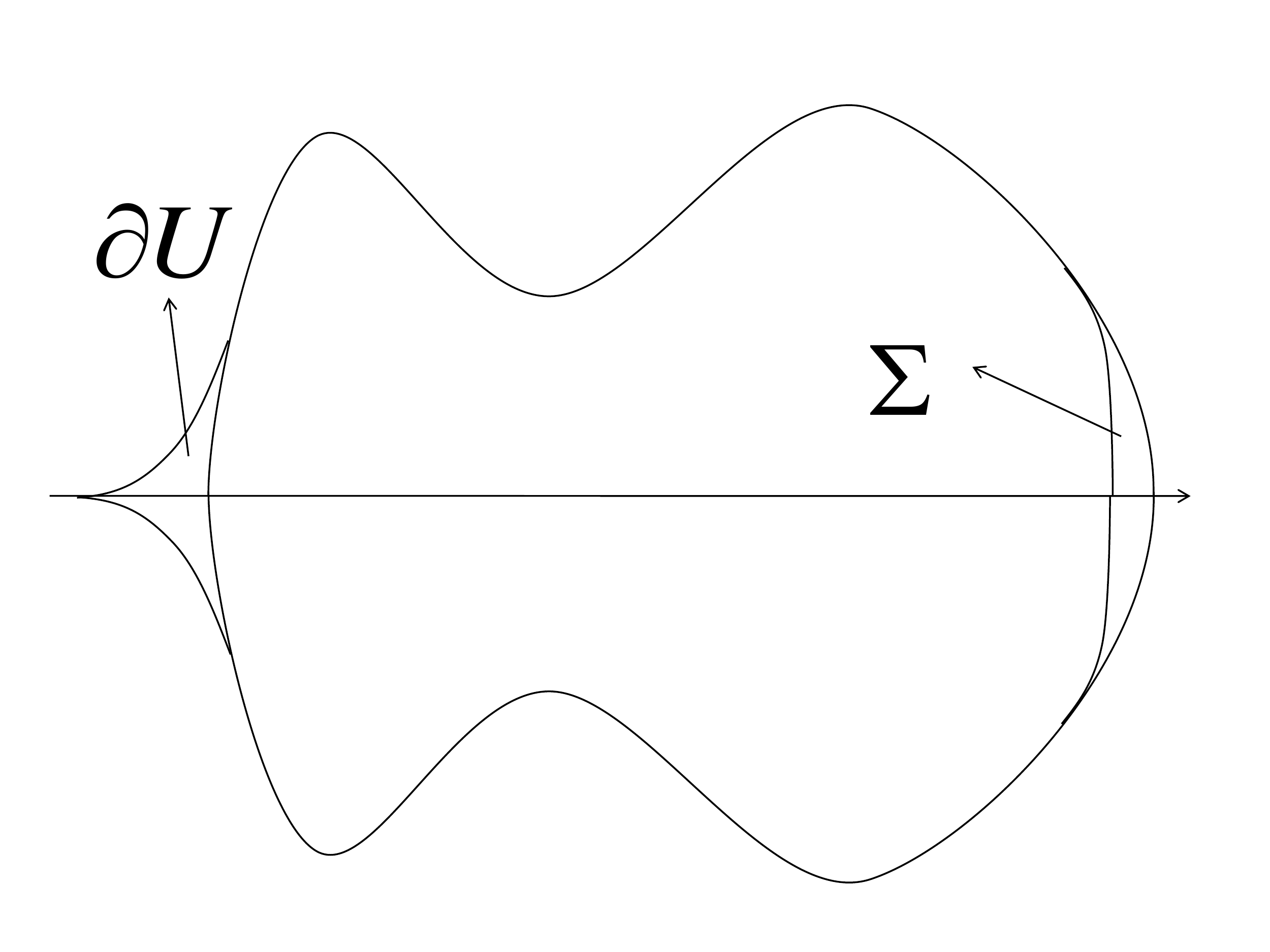}
		\vskip 0pt
		\caption{Proof of Proposition \ref{pro:sigma1}}
        \label{fig:sigma1}
	\end{center}
\end{figure}

Let $\phi^{\alpha}:\mathbb{R}^{n+1}\rightarrow\mathbb{R}^{n+1} (\alpha\in\mathbb{R})$, $t\in(-\delta,\delta)$ be the flow generated by vector field $X$ on $\mathbb{R}^{n+1}$, that is,
$$
\left\{
\begin{array}{lcl}
\dis{\frac{d\phi^{\alpha}(P)}{d\alpha}=X(\phi^{\alpha})},\ P\in \Sigma,\\
\phi^0(P)=P,\ \ \ \ \ P\in \Sigma.
\end{array}
\right.
$$

We denote $\sigma(P,s):=\phi^{s}(P)$. As in Section 5, suppose $\Gamma(t)\subset V$ $(0<t<T)$ are smooth hypersurfaces such that $\sigma^{-1}(\Gamma(t))$ is the graph $u(\cdot,t)$ for $u:\Sigma\times[0,T)\rightarrow\mathbb{R}$. Let $z_1,z_2,\cdots,z_n$ be local coordinates on an open subset of $\Sigma$. If $\Gamma(t)$ evolving by $V=-\kappa+A$, then in these coordinates $u$ satisfies the following parabolic equation
\begin{equation}\label{eq:para}
\frac{\partial u}{\partial t}=a_{ij}(z,u,\nabla u)\frac{\partial^2u}{\partial z_i\partial z_j}+b(z,u,\nabla u).
\end{equation}

\begin{figure}[htbp]
	\begin{center}
            \includegraphics[height=4cm]{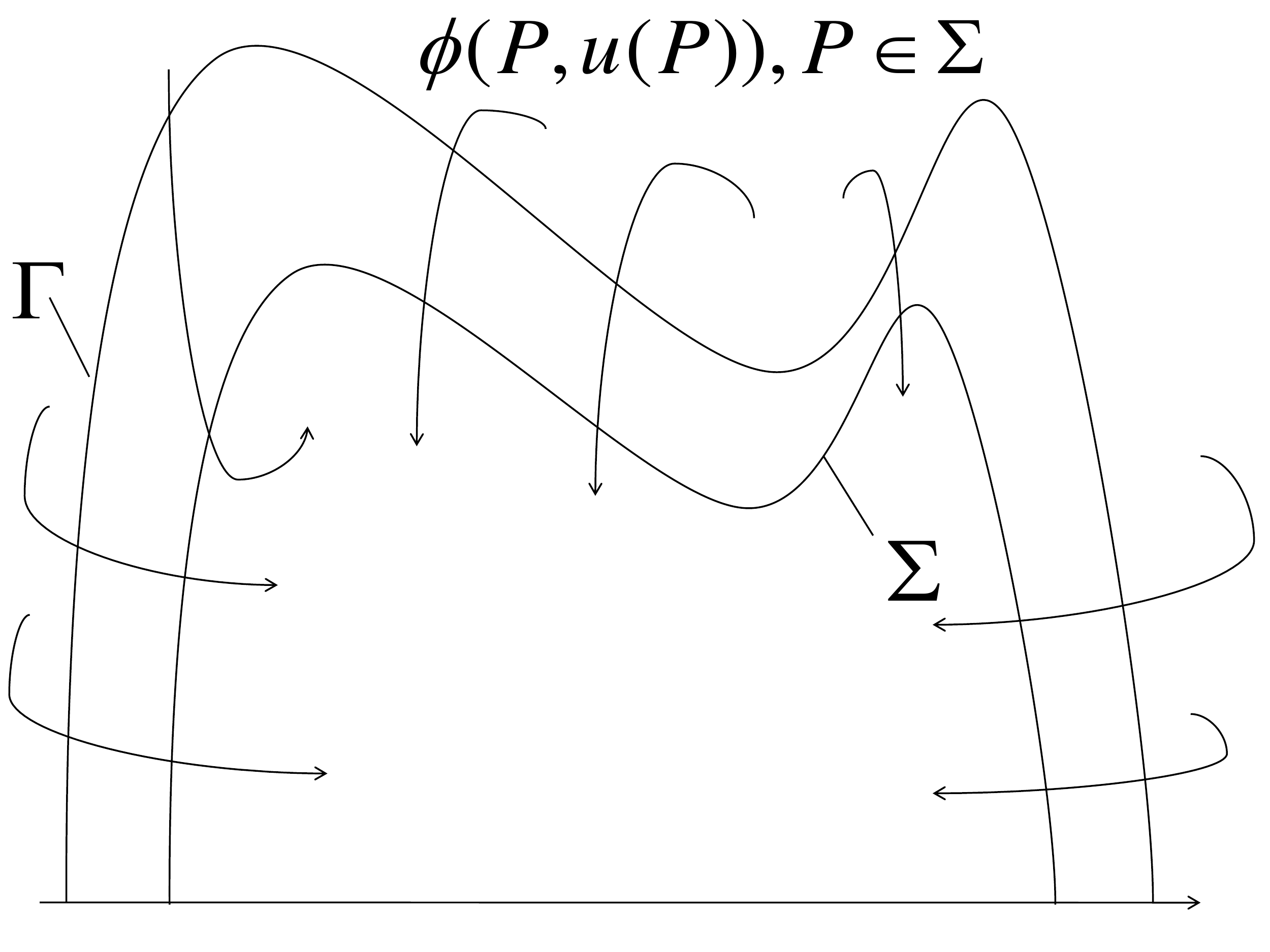}
		\vskip 0pt
		\caption{The transportation from $\Sigma$ to $\Gamma$}
        \label{fig:charact}
	\end{center}
\end{figure}
For example, on $\Delta_a$, by calculation, $\sigma(y_1,,y_2,\cdots,y_n,s)=(a-s,y_1,y_2,\cdots,y_n)$. Then $u$ satisfies the "$-$" condition of (\ref{eq:graph}).
\begin{prop}\label{pro:uniq} For $n\geq1$, let $\Gamma_1(t)$, $\Gamma_2(t)$ $(0\leq t<T)$ be two families of smooth hypersurface smooth and $\sigma^{-1}(\Gamma_j(t))$ be the graph of $u_j(\cdot,t)$ for certain $u_j\in C(\Sigma\times[0,T))$. Assume that the $u_j$ are smooth on $\Sigma\times(0,T)$ as well as on $\Sigma\setminus(\Delta_a\cup\Delta_b)\times[0,T)$. Then if the $\Gamma_j(t)$ evolve by $V=-\kappa+A$ and if $\Gamma_1(0)=\Gamma_2(0)$, then there holds $\Gamma_1(t)=\Gamma_2(t)$ for $0<t<T$.
\end{prop}
We use the same method in \cite{AAG}. 
The proof is similar as in Proposition \ref{pro:uniq2}. Here we omit it.

\begin{thm}\label{thm:partialUmeancurvature}
If $U$ is an $\alpha$-domain with a smooth boundary, let $D(t)$ and $E(t)$ be the open and closed evolutions of $V=-\kappa+A$ with $D(0)=U$ and $E(0)=\overline{U}$. Then there exists $T>0$ such that $\partial D(t)$ and $\partial E(t)$ are smooth hypersurfaces for $0<t\leq T$ and $\partial D(t)=\partial E(t)$. Moreover, denoting $\Sigma(t)=\partial D(t)=\partial E(t)$, $\Sigma(t)$ can be written into $\Sigma(t)=\{(x,y)\in\mathbb{R}\times\mathbb{R}^n\mid |y|=u(x,t), a(t)\leq x\leq b(t)\}$ and $(u,a,b)$ satisfies
\begin{equation}
\left\{
\begin{array}{lcl}
\dis{u_t=\frac{u_{xx}}{1+u_x^2}-\frac{n-1}{u}+A\sqrt{1+u_x^2}},\ x\in(a(t),b(t)),\ 0<t< T,\\
u(a(t),t)=0,\ u(b(t),t)=0,\ 0\leq t< T,\\
u_x(a(t),t)=\infty,\ u_x(b(t),t)=-\infty,\ 0\leq t<T,\\
\end{array}
\right.\tag{**}
\end{equation}
\end{thm}
\begin{proof}
We only give the sketch of the proof. By approximate argument similarly as in Lemma 6.2 and Lemma 6.4, $\partial D(t)$ and $\partial E(t)$ are smooth hypersurfaces and can be represented by $\sigma(P,u_j(P))$, for some $u_j$, $j=1,2$. Then we can use Proposition \ref{pro:uniq} to prove $\partial D(t)=\partial E(t)$. Therefore $\Gamma(t)=\partial E(t)$ can be represented by $\Gamma(t)=\{(x,y)\in\mathbb{R}\times\mathbb{R}^n\mid |y|=u(x,t), a(t)\leq x\leq b(t)\}$. Using Theorem \ref{thm:openevolutionmeancurvature}, $(u,a,b)$ satisfies (**).
\end{proof}

{\bf Conflict of interest.} 
We declare that we do not have any commercial or associative interest that represents a conflict of interest in connection with the work submitted.

\end{document}